\documentclass[11pt]{article}
\usepackage{bbm}
\usepackage{eqnarray,amsmath,amsfonts,amsthm,mathrsfs}
\usepackage{color}
\usepackage{bm}
\usepackage{amssymb}
\usepackage{graphicx}
\usepackage{epsfig}
\usepackage{epsf}
\usepackage{float}
\usepackage{epstopdf}
\usepackage{lipsum}
\usepackage{subcaption}
\usepackage{amsfonts,amsmath,amsthm,amssymb,graphicx,float,fancyhdr,multirow,hyperref}
\usepackage{booktabs,longtable,authblk}
\usepackage{mathrsfs,hhline}
\usepackage{algorithm}
\usepackage{algorithmic}
\allowdisplaybreaks[4]

\usepackage{fancyhdr}
\usepackage[top=2.5cm, bottom=2.5cm, left=3cm, right=3cm]{geometry}
\usepackage{graphicx}
\newtheorem{theorem}{Theorem}
\newtheorem{assumption}{Assumption}

\newtheorem{definition}{Definition}

\newtheorem{lemma}{Lemma}
\newtheorem{proposition}{Proposition}
\newtheorem{remark}{Remark}

\allowdisplaybreaks[4]

\numberwithin{equation}{section}

\title{An Efficient Deep Learning Approach for Approximating Parameter-to-Solution Maps of PDEs $^\dag$\footnotetext{\dag~The work described in this paper is supported by the National Natural Science Foundation of China (Grants Nos.12171039 and 12061160462) and Shanghai Science and Technology Program [Project No. 21JC1400600]. The corresponding author is Lei Shi.}}

\author{Guanhang Lei}
\author{Zhen Lei}
\author{Lei Shi}
\author{Chenyu Zeng}

\affil{School of Mathematical Sciences, \linebreak
Shanghai Key Laboratory for Contemporary Applied Mathematics, \linebreak
Fudan University, Shanghai, 200433, P. R. China \linebreak
Email:ghlei21@m.fudan.edu.cn \linebreak
\{zlei, leishi, cyzeng19\}@fudan.edu.cn}

\begin{document}
	\maketitle
\begin{abstract}
	In this paper, we consider approximating the parameter-to-solution maps of parametric partial differential equations (PPDEs) using deep neural networks (DNNs). We propose an efficient approach combining reduced collocation methods (RCMs) and DNNs. In the approximation analysis section, we rigorously derive sharp upper bounds on the complexity of the neural networks. These bounds only depend on the reduced basis dimension rather than the high-fidelity discretization dimension, thereby theoretically guaranteeing the computational efficiency of our approach. In numerical experiments, we implement the RCM using radial basis function finite differences (RBF-FD) and proper orthogonal decomposition (POD), and propose the POD-DNN algorithm. We consider various types of PPDEs and compare the accuracy and efficiency of different solvers. The POD-DNN has demonstrated significantly accelerated inference speeds compared with conventional numerical methods owing to the offline-online computation strategy. Furthermore, by employing the reduced basis methods (RBMs), it also outperforms standard DNNs in computational efficiency while maintaining comparable accuracy. 
\end{abstract}
	

\section{Introduction}\label{section: Introduction}

Parametric partial differential equations (PDEs) are widely employed in depicting or modeling intricate phenomena within various fields of science and engineering. These equations use parameters to account for variations in physical properties, geometric configurations, and initial or boundary conditions. For instance, within the realm of fluid mechanics, parametric PDEs are pivotal in representing how both liquids and gases move, enabling the simulation of diverse scenarios, like turbulence, wave propagation, and the interaction between different fluid layers, with equation parameters used to describe fluid attributes like viscosity and density, as well as external forces and boundary conditions. Solving parametric PDEs directly is possible via established discretization techniques, including the finite element, spectral, and finite volume methods. However, these methods often require fine discretization (e.g., fine meshing of the solution region) and solving linear systems with thousands of degrees of freedom to obtain sufficiently accurate solutions, resulting in unmanageable demands on computational resources. Additionally, in many practical applications in science and engineering, we often need to solve parametric PDEs in real-time against multiple queries, necessitating repeated and expensive computations, thus further increasing computational costs. In such cases, there is a need to develop efficient numerical methods that allow for rapid simulation of solutions when the equation parameters change.

The reduced basis method (RBM) is widely applied to the numerical simulations of large-scale problems involving parametric PDEs \cite{Quarteroni2016Reduced}. Intuitively, the solutions to parametric PDEs are often determined by several equation parameters, indicating that the solutions associated with all possible parameters constitute a low-dimensional manifold within the solution space. RBM partitions the numerical solution process into offline/online stages. In the offline phase, which typically incurs a relatively higher computational cost, we fully exploit this low-dimensional nature and use some precomputed solutions, i.e., snapshots, to construct what is known as the reduced basis space, a low-dimensional approximation space. During the online phase, we leverage this set of basis functions to form a small-scale algebraic system that can be solved in real time, thus enabling efficient numerical solving of the equation when input parameters change. Based on the low-dimensional assumption of the solution manifold, the solution for any query parameter values can be approximated by the span of these snapshots in the low-dimensional space. Once the low-dimensional space is established, the solution for new parameter values is efficiently solved in the online phase.

Given new parameters, RBM calculates the coefficients of the associated solution under the span of this reduced basis in the online phase. This parameter-coefficients relation can be seen as the so-called parametric map. Although we only need to solve a relatively small-scale linear system in the online phase, the computational cost remains prohibitively expensive for some large-scale practical problems. To further decrease the response time and enhance the efficiency of numerical simulations, we can train deep neural networks (DNNs) to directly learn the parametric map, thereby transforming the online calculation into a feedforward computation within the networks. The DNN-based approach retains the offline-online decomposition feature, and integrating DNN with RBM significantly diminishes the computational expenses during the online phase, albeit at the cost of introducing an additional stage of neural network training in the offline phase. This additional step usually consumes less time compared to constructing the low-dimensional approximation space. Moreover, this end-to-end DNN methodology facilitates the parallel processing of multiple parameters, offering further benefits when simultaneously handling various parameters.

Constructing an efficient RBM is central to designing numerical solutions for parametric maps. The classical RBM was initially integrated with the Galerkin method to solve parametric PDEs that exhibit affine dependence on parameters. Subsequently, reduced collocation methods (RCMs) were developed based on the collocation approach. RCMs offer greater flexibility and efficiency, reducing the online computational cost of solving nonlinear parametric PDEs using the Galerkin approach. Traditional collocation methods, such as pseudospectral methods \cite{Hesthaven2007Spectral}, often require highly structured grids, which may be inconvenient for solving PDEs with irregularly shaped domains. In numerical experiments, to further address the challenges of solving parametric PDEs in irregular domains, motivated by previous work \cite{Chen2016Reduced}, we utilize radial basis functions (RBFs) to construct a meshless RCM. RBF methods have been extensively employed for scattered data interpolation and approximation in high dimensions \cite{Kansa1990Multiquadrics, Wendland2004Scattered, Fornberg2015Solving}. Unlike traditional pseudospectral methods, RBF methods are collocation methods implemented on scattered sets of collocation sites and are not confined to any particular geometric structure. Typically, RBF methods approximate differential operators using global stencils. \cite{AndreiI2000using} first introduces a local RBF technique, which is inspired by the finite difference (FD) methods and referred to as the RBF-FD method. To construct the low-dimensional reduced basis space. Popular approaches in this step include the greedy algorithm \cite{Prudhomme2001Reliable} and the proper orthogonal decomposition (POD) \cite{Christensen.1999Evaluation, Rathinam2003New, Kunisch2001Galerkin}. The greedy algorithm selects one optimal parameter and its associated snapshot by a posteriori error estimate at each iterative step. POD is based on the classic singular value decomposition (SVD) of a matrix, which directly extracts the main part of the snapshot matrix.

Deep learning techniques have been successfully applied to develop efficient methods for solving PDEs. For example, DNN has been effectively implemented across various numerical PDE challenges. The deep Ritz method (DRM) \cite{E2018Deep} and physics-informed neural networks (PINNs) \cite{Raissi2019Physicsinformed, Sirignano2018DGM} have been designed to train neural networks with functional-type losses to approximate underlying PDEs' solutions directly. Neural operator methods like DeepONet \cite{Lu2021Learning}  and Fourier neural operator (FNO) \cite{Li2021Fourier} employ a supervised learning framework in infinite-dimensional input and output function spaces. They utilize various encoding and decoding techniques to simplify the operator regression task into learning within finite-dimensional tensor spaces. Neural operators are trained to learn the PDE solution operators, eliminating the need to train a new neural network for each new input function. Deep learning methods have also been extensively investigated and widely utilized in parametrized systems \cite{Hesthaven2018Non-intrusive, Tripathy2018Deepa, DalSanto2020Dataa, Khoo2021Solvinga, Geist2021Numerical, Kutyniok2022Theoreticala, Lei2022Solving, Franco2023deep}.

In this paper, we approximate the parametric maps of parametric PDEs through a synergistic combination of the RCM and DNNs. The main contribution of this article is the rigorous derivation of the algorithm's complexity upper bounds under reasonable assumptions, as seen in Theorem \ref{main theorem: approximate parametric map}. Specifically, we establish the upper bounds for the depth and the number of non-zero parameters of the ReLU DNNs when approximating the map $\mu \mapsto \mathbf{c}_\mu$ with precision $\varepsilon > 0$. Here, $\mu$ denotes the parameter of the parametric PDE, and $\mathbf{c}_\mu$ represents the coefficient vector of the solution $\mathbf{u}_\mu$ within the span of the reduced basis. Our derived upper bounds indicate that the network's depth is approximately $\mathcal{O}(\log_2(1/\varepsilon )\log_2(\log_2(1/\varepsilon)))$, and the number of non-zero parameters is $\mathcal{O}(d^3\log^2_2(1/\varepsilon))$, where $d$ is the dimension of the reduced basis. This bound on the number of non-zero parameters is better than the previous result $\mathcal{O}(d^3\log^2_2(1/\varepsilon)\log^2_2(\log_2(1/\varepsilon)))$ in \cite{Kutyniok2022Theoreticala}. This improvement is attributed to the application of a different Neumann series expansion in the proof, which is motivated by our previous work on ReQU DNNs \cite{Lei2022Solving}. Our complexity analysis theoretically guarantees the efficiency of the deep learning algorithm. 

In numerical experiments, we consider an explicit discretization method and RBM to propose the POD-DNN algorithm. POD-DNN leverages the offline-online computational strategy inherent in DNNs, showcasing reduced computational costs during the online phase compared to the RBF-based RCM proposed in \cite{Chen2016Reduced}, as evidenced by our numerical experiments. Furthermore, POD-DNN significantly reduces the output dimension of DNNs, facilitating easier network training, convergence, and inference processes. We first test POD-DNN in typical linear PPDEs where the parameter space is low-dimensional. We also test our algorithm on more general PPDEs, including nonlinear PDEs, time-dependent PDEs, and PDEs with infinite-dimensional functional parameters.

The rest of the paper is organized as follows. In Section \ref{section: Preliminaries and Main Results}, we briefly review the PPDE problem, parametric maps, and the combination of RBM and DNN to numerically simulate parametric maps. We then present our assumptions and main result, delineated as Theorem \ref{main theorem: approximate parametric map}. In Section \ref{section: Theoretic Analysis}, we commence with basic definitions and operations relevant to neural networks, leading to the proof of our main theorem. The experimental validation of our algorithm is detailed in Section \ref{section: Numerical Experiments}, where we also compare the numerical results with conventional methods, non-neural network-based offline-online algorithms, and neural operators in different scenarios.

\section{Preliminaries and Main Results}\label{section: Preliminaries and Main Results}
\subsection{Parametric PDE and Reduced Basis Method}
We first consider the following general parametric PDE
\begin{equation}\label{PPDE}
    \mathcal{L}(u_\mu; \mu) = 0 \quad \text{in } H
\end{equation}
where $\mu= (\mu_1, \ldots, \mu_p) \in \mathcal{D} \subset \mathbb{R}^p$ with $\mathcal{D}$ a prescribed $p$-dimensional real parameter domain, $\mathcal{L}(\cdot; \mu)$ is a parametrized operator that depends on the parameter $\mu$ and represents a parametrized PDE, $u_\mu \in H$ is the unknown solution defined on a physical domain $\Omega$ and associated with the parameter $\mu$, and $H$ is a Hilbert space. The PPDE above is presented in the strong form, and we can consider its weak (variational) form:
\begin{equation}\label{PPDE: weak form}
	\langle \mathcal{L}(u_\mu; \mu), v \rangle = 0, \forall v \in H.
\end{equation}

For numerical computation, we consider a discretization of $H$ represented by a finite-dimensional subspace, typically a high-fidelity discretization $H^{\mathrm{h}} \subset H$ which is potentially high-dimensional such that the corresponding high-fidelity solution (or snapshot) $u^{\mathrm{h}}_{\mu} \in H^{\mathrm{h}}$ can reach an arbitrary small error $\varepsilon>0$:
\[
	\sup_{\mu \in \mathcal{D}} \|u_\mu - u^{\mathrm{h}}_{\mu}\|_{H} \leq \varepsilon.
\]
Then one can design a numerical algorithm to approximate the so-called parametric map $\mu \mapsto u^{\mathrm{h}}_{\mu}$. Let $D = D(\varepsilon)$ denote the dimension of the high-fidelity space $H^{\mathrm{h}}$. In practice, $D$ is typically a large number that leads to prohibitively high numerical computation costs. The starting point of RBM lies in the property that, under some regularity assumptions of PPDEs, there exists a reduced basis space $H^{\mathrm{rb}}$ with much lower dimension $d = d(\varepsilon') \ll D$ while the approximation error $\varepsilon' \geq \varepsilon$ remains relatively small:
\[
	\sup_{\mu \in \mathcal{D}} \inf_{u \in H^{\mathrm{rb}}} \|u - u_\mu\|_H \leq \varepsilon'.
\]
This property can be established using Kolmogorov $N$-width in the context of PPDEs where the solution manifold is usually a low-dimensional compact subset of $H$, see \cite{Kutyniok2022Theoreticala}. Let $(\psi_i)_{i=1}^d$ denote the basis of $H^{\mathrm{rb}}$. $(\psi_i)_{i=1}^d$ can be generated from linear combinations of a basis $\{e_j\}_{j=1}^D$ of the high-fidelity space by a transformation matrix $V \in \mathbb{R}^{D \times d}$:
\[
	\psi_i = \sum_{j=1}^D V_{j,i}e_j, \quad i = 1,\ldots, d.
\]
The numerical algorithm consequently shifts to computing the reduced basis solution 
\[
	u^{\mathrm{rb}} = \sum_{i=1}^d \left(\mathbf{c}_\mu\right)_i \psi_i
\]
by approximating the new parametric map $\mu \mapsto \mathbf{c}_\mu \in \mathbb{R}^d$. This map is $\mathbb{R}^p \to \mathbb{R}^d$ and hence can be approximated by training neural networks $\mathbf{NN}:\mathbb{R}^p \to \mathbb{R}^d$.

Previous works \cite{Kutyniok2022Theoreticala, Lei2022Solving} consider the combination of DNN and RBM in the Galerkin framework, where the high-fidelity solutions can be obtained from the finite element method. Assumptions for their DNN approximation analysis are on the approximability of the stiffness matrix and the discretized right-hand side with respect to the Galerkin reduced basis. In this paper, motivated by \cite{Chen2013Reduceda, Chen2016Reduced}, we consider RBM in the collocation framework, also known as the so-called RCM. We now consider the following linear PPDE:
\begin{equation}\label{linear PPDE}
	\left\{ 
        \begin{aligned}
            &\mathcal{L}(\mu)u_\mu(x)=f(x; \mu), \quad x \in \Omega \subset \mathbb{R}^n, \\
            &\mathcal{B}(\mu)u_\mu(x)=g(x; \mu), \quad x \in \partial\Omega,
        \end{aligned}
    \right.
\end{equation}
where $x \in \mathbb{R}^n$ represents the spatial variable, $\mathcal{L}(\mu)$ is a parametrized differential operator, and $\mathcal{B}(\mu)$ is the boundary condition. In collocation framework, we assume that the physical domain $\Omega$ is discretized onto a set of $N$ collocation points $X = \left\{x_i\right\}_{i=1}^N \subset \bar{\Omega}$, which contains interior and boundary nodes, i.e., $X= X_{\Omega}\cup X_{\partial \Omega}$ where $X_{\Omega}=\left\{x_1, \ldots, x_{N_I}\right\} \subset \Omega$ and $X_{\partial \Omega}=\left\{x_{N_I+1}, \ldots,x_N\right\} \subset \partial \Omega$. The differential operator $\mathcal{L}(\mu)$ and boundary condition $\mathcal{B}(\mu)$ are also discretized into one compact matrix form $A_\mu \in \mathbb{R}^{N \times N}$ through some appropriate scheme. We find the solution to the equation
\begin{equation}\label{linear equation system}
	A_\mu\mathbf{u}_\mu = 
    \begin{bmatrix}
        \mathbf{f}_\mu \\
        \mathbf{g}_\mu
    \end{bmatrix},
\end{equation}
where $\mathbf{f}_\mu = \left(f(x_1), \ldots, f(x_{N_I})\right)^T$ and $\mathbf{g}_\mu = \left(g(x_{N_I+1}), \ldots, g(x_{N})\right)^T$ are discretized right-hand side function. The solution is then the discretized approximated solution on collocation points: $\mathbf{u}_\mu = \left(u(x_1), \ldots, u(x_{N})\right)^T$, which is also the high-fidelity solution for parameter $\mu$ in the context of RBM. 

To construct a reduced basis, we sample $n_s$ parameter samples $\Xi = \left\{\mu^1, \ldots, \mu^{n_s}\right\}$ from the parameter space (typically by a random generation or a uniform lattice) and compute the associated high-fidelity snapshots $\left\{\mathbf{u}_{\mu^1}, \ldots, \mathbf{u}_{\mu^{n_s}}\right\}$ by \eqref{linear equation system}. Then we may adopt the greedy algorithm or principal component analysis to generate the reduced basis. For the greedy algorithm \cite{Chen2013Reduceda, Chen2016Reduced, Quarteroni2016Reduced, Chen2018Greedy}, we sequentially select optimal parameters by some posterior error estimation criterion until the posterior error is sufficiently small. The associated snapshots of these selected parameters are then modified by a Gram-Schmidt transformation and form the reduced basis. For principal component analysis, we employ the POD on the high-fidelity snapshot matrix $S = \left[\mathbf{u}_{\mu^1} \vert \cdots \vert \mathbf{u}_{\mu^{n_s}}\right] \in \mathbb{R}^{N \times n_s}$. The first few left singular vectors form the orthogonal reduced basis.

Now, assume that we obtain the orthogonal reduced basis $\{\boldsymbol{\zeta}_1, \ldots, \boldsymbol{\zeta}_{d}\}$ and the corresponding transformation matrix $V = [\boldsymbol{\zeta}_1 \vert \cdots \vert \boldsymbol{\zeta}_{d}] \in \mathbb{R}^{N \times d}$. Then for any new $\mu \in \mathcal{D}$, we can approximate $\mathbf{u}_{\mu}$ by a span of the reduced basis, i.e., 
\[
    \mathbf{u}_{\mu} \approx \sum_{i=1}^{d} \left(\mathbf{c}_\mu\right)_i\boldsymbol{\zeta}_i = V\mathbf{c}_{\mu}.
\]
Substitute it into \eqref{linear equation system}, we have
\begin{equation}\label{B_mu}
	\begin{aligned}
		B_{\mu}\mathbf{c}_{\mu} = 
		\begin{bmatrix}
			\mathbf{f}_\mu \\
			\mathbf{g}_\mu
		\end{bmatrix}, \quad
		B_{\mu} := A_\mu V \in \mathbb{R}^{N \times d}.
	\end{aligned}
\end{equation}
This is typically an overdetermined system, hence we seek the least square solution, which is given by the Moore-Penrose inverse
\begin{equation}\label{Moore-Penrose inverse}
    B_{\mu}^{\dagger}=\left(B_{\mu}^T B_{\mu}\right)^{-1} B_{\mu}^T.
\end{equation}
Here, we assume that the reduced basis is well-chosen such that $B_{\mu}$ is column full-rank. The least square solution reads 
\begin{equation}\label{least square solution}
	\mathbf{c}_{\mu} = B_{\mu}^{\dagger}
	\begin{bmatrix}
        \mathbf{f}_\mu \\
        \mathbf{g}_\mu
    \end{bmatrix}
	= \left(B_{\mu}^T B_{\mu}\right)^{-1} B_{\mu}^T 
	\begin{bmatrix}
        \mathbf{f}_\mu \\
        \mathbf{g}_\mu
    \end{bmatrix}.
\end{equation}
\eqref{least square solution} gives analysis implicit form of the parametric map $\mu \mapsto \mathbf{c}_{\mu}$. 

\subsection{Assumptions and Main Results}
To approximate the parametric map $\mu \mapsto \mathbf{c}_{\mu}$ using DNNs, \eqref{least square solution} motivates us to design DNNs to approximate matrix inversion and matrix multiplication, which aligns with our proof of the main theorem. It remains to approximate the implicit maps $\mu \mapsto B_{\mu}, \mu \mapsto \mathbf{f}_\mu$ and $\mu \mapsto \mathbf{g}_\mu$. Before we state the assumptions on these approximations, we first introduce the mathematical definition of deep ReLU networks, incorporating symbols and conventions that will be used consistently or referenced from the literature \cite{Kutyniok2022Theoreticala}.

\begin{definition}
	Let $L, n_0\in \mathbb{N}$. A neural network $\Phi$ with input dimension $n_0$ and number of layers $L = L(\Phi)$ is a matrix-vector sequence
	\[
		\Phi=\left(\left(W_{1}, b_{1}\right), \cdots,\left(W_{L}, b_{L}\right)\right),
	\]
	where $W_{i} \in \mathbb{R}^{n_{i} \times n_{i-1}}$ are weight matrices, $b_{i} \in \mathbb{R}^{n_{i}}$ are bias vectors, and $n_{i} \in \mathbb{N}$ are width. Let $\sigma: \mathbb{R} \rightarrow \mathbb{R}$ be ReLU activation function, i.e., $\sigma(x) = \max\{0,x\}$. We define the realization of the network 
	\[
		R(\Phi): \mathbb{R}^{n_0} \rightarrow \mathbb{R}^{n_{L}}, \quad R(\Phi)(x)=x_{L},
	\]
	where $x_{L}$ is expressed as follows:
	\[
		\left\{\begin{array}{l}
			x_{0}:=x, \\
			x_{i}:=\sigma\left(W_{i} x_{i-1}+b_{i}\right), \quad i = 1,2, \ldots, L-1, \\
			x_{L}:=W_{L} x_{L-1}+b_{L}.
		\end{array}\right.
	\]
	Here, the ReLU function $\sigma$ operates on the vector in an element-wise manner. Define the number of non-zero parameters in $i$-th layer $M_{i}(\Phi):=\left\|W_i\right\|_{0}+\left\|b_i\right\|_{0}$ and the total number of non-zero parameters $M(\Phi):=\sum_{i=1}^{L} M_{i}(\Phi)$. 
\end{definition}
\begin{definition}
	Let $A \in \mathbb{R}^{n \times k}$. We denote
	\[
		\mathrm{vec}(A):=\left(A_{1,1}, \cdots, A_{n, 1}, \cdots, A_{1, k}, \cdots, A_{n, k}\right)^{T} \in \mathbb{R}^{nk}.
	\]
	Moreover, for a vector $v=\left(v_{1,1}, \cdots, v_{n, 1}, \cdots, v_{1, k}, \cdots, v_{n, k}\right)^{T} \in \mathbb{R}^{nk}$, we set
	\[
		\mathrm{matr}(v):=\left(v_{i, j}\right)_{i=1, \ldots, n, j=1, \ldots, k} \in \mathbb{R}^{n \times k}.
	\]
\end{definition}
\begin{assumption}\label{assumption: spectrum of BTB}
	The spectrum $\Lambda(B_{\mu}^T B_{\mu}) \subset \left[\beta^2, \alpha^2\right]$ with $0<\beta^2<\alpha^2$ for all $\mu \in \mathcal{D}$ and possible reduced basis $V$. The norm $\|(\mathbf{f}^T_\mu, \mathbf{g}^T_\mu)^T\|_2 \leq \gamma$ with $\gamma > 0$ for all $\mu \in \mathcal{D}$. 
\end{assumption}
\begin{assumption}\label{assumption: mu to B}
	For any $\varepsilon>0$ and possible reduced basis $V$, there exists a ReLU neural network $\Phi_{V, \varepsilon}^{B}$ with $p$-dimensional input and $Nd$-dimensional output such that
	\[
		\sup_{\mu \in \mathcal{D}}\left\|B_{\mu}-\mathrm{matr}\left(R\left(\Phi_{V, \varepsilon}^{B}\right)(\mu)\right)\right\|_{2} \leq \varepsilon,
	\]
   	which implies another ReLU neural network $\Phi_{V, \varepsilon}^{B^T}$ such that 
	\[
		\sup_{\mu \in \mathcal{D}}\left\|B^T_{\mu}-\mathrm{matr}\left(R\left(\Phi_{V, \varepsilon}^{B^T}\right)(\mu)\right)\right\|_{2} \leq \varepsilon.
	\]
	We set $L(B; V, \varepsilon) := L(\Phi_{V, \varepsilon}^{B}) = L(\Phi_{V, \varepsilon}^{B^T})$ and $M(B; V, \varepsilon) := M(\Phi_{V, \varepsilon}^{B}) = M(\Phi_{V, \varepsilon}^{B^T}).$
\end{assumption}
\begin{assumption}\label{assumption: mu to fg}
	For any $\varepsilon>0$, there exists a ReLU neural network $\Phi_{\varepsilon}^{\mathbf{fg}}$ with $p$-dimensional input and $N$-dimensional output such that
	\[
		\sup_{\mu \in \mathcal{D}}\left\|(\mathbf{f}^T_\mu, \mathbf{g}^T_\mu)^T - R\left(\Phi_{\varepsilon}^{\mathbf{fg}}\right)(\mu)\right\|_{2} \leq \varepsilon.
	\]
	We set $L(\mathbf{fg}; \varepsilon) := L(\Phi_{\varepsilon}^{\mathbf{fg}})$ and $M(\mathbf{fg}; \varepsilon) := M(\Phi_{\varepsilon}^{\mathbf{fg}}).$
\end{assumption}
\begin{remark}\label{remark: assumption}
	These assumptions can be fulfilled in a variety of cases. Recall the definition of $B_{\mu}$ in \eqref{B_mu}, we note that approximating the map $\mu \mapsto B_{\mu}$ is essentially equivalent to approximating the map of stiffness matrix $\mu \mapsto A_\mu$. These maps, together with $\mu \mapsto (\mathbf{f}^T_\mu, \mathbf{g}^T_\mu)^T$ in \autoref{assumption: mu to fg}, can be easily constructed or approximated by neural networks if the parameter dependency of the PPDE \eqref{linear PPDE} exhibits a simple form, such as a linear or rational nature. See the discussion in our numerical experiments for various equations.
\end{remark}
With these assumptions, we can now construct ReLU DNNs that approximate the parametric map $\mu \mapsto \mathbf{c}_\mu$. We present our main result.
\begin{theorem}\label{main theorem: approximate parametric map}
	Suppose that \autoref{assumption: spectrum of BTB}, \autoref{assumption: mu to B}, and \autoref{assumption: mu to fg} hold. For any $\varepsilon \in (0,1/4)$ and possible reduced basis $V$, there exists a ReLU neural network $\Phi_{\mathrm{para};\varepsilon}^{\mathbf{c}}$ with $p$-dimensional input and $d$-dimensional output that satisfies the following properties.
	\begin{enumerate}
		\item $\displaystyle \sup_{\mu \in \mathcal{D}}\left\|\mathbf{c}_\mu - R\left(\Phi_{\mathrm{para};\varepsilon}^{\mathbf{c}}\right)(\mu)\right\|_{2} \leq \varepsilon.$
		\item There exists a constant $C_{L}^{\mathbf{c}}$ only depending on $\alpha,\beta,\gamma$, such that
		\[
			\begin{aligned}
				L(\Phi_{\mathrm{para};\varepsilon}^{\mathbf{c}}) \leq{}& C^{\mathbf{c}}_{L}\cdot\log_2(\log_2(1/\varepsilon))\cdot(\log_2(1/\varepsilon) +  \log_2 d) \\
				&+C^{\mathbf{c}}_{L} \log_2 N + L(B; V, \varepsilon_1) + L(B; V, \varepsilon_{3}) + L(\mathbf{fg}; \varepsilon_2),
			\end{aligned}
		\]
		\item There exists a constant $C_{M}^{\mathbf{c}}$ only depending on $\alpha,\beta,\gamma$, such that
		\[
			\begin{aligned}
				&M(\Phi_{\mathrm{para}; \varepsilon_{1:5}}^{\mathbf{c}}) \\
				\leq{}& C^{\mathbf{c}}_{M}N d^2 \cdot (\log_2(1/\varepsilon) + \log_2 N) \\
				&+ C^{\mathbf{c}}_{M}d^2 \log_2(1/\varepsilon) \cdot (\log_2(\log_2(1/\varepsilon)) + d)\cdot(\log_2(1/\varepsilon) + \log_2 d) \\
				&+ C^{\mathbf{c}}_{M}d^2  \cdot \log_2(\log_2(1/\varepsilon)) \cdot \log_2 N \\
				&+ 16M(B; V, \varepsilon_1) + 16M(\mathbf{fg}; \varepsilon_2) + 8M(B; V, \varepsilon_{3})  \\
				&+80NdL(B; V, \varepsilon_1) + 80NdL(\mathbf{fg}; \varepsilon_2) + 16d^2L(B; V, \varepsilon_{3}).
			\end{aligned}
		\]
	\end{enumerate}
	Here $\displaystyle \varepsilon_1:= \min \left\{\frac{\varepsilon\beta^2}{12\gamma}, \frac{\beta \sqrt{\varepsilon}}{4}, \frac{\alpha}{4}, \frac{\sqrt{\alpha \gamma}}{2}\right\}, \varepsilon_2:= \min \left\{\frac{\varepsilon\beta^2}{12\alpha}, \frac{\beta \sqrt{\varepsilon}}{4}, \frac{\gamma}{4}, \frac{\sqrt{\alpha \gamma}}{2}\right\}$ and \\$\displaystyle \varepsilon_{3}:= \min\left\{\frac{\varepsilon \beta^4}{144\alpha^2 \gamma}, \frac{\beta^2\sqrt{6\alpha\gamma\varepsilon}}{24\alpha\gamma}, \frac{\beta^2}{12\alpha}, \frac{\beta}{3}\right\}$.
\end{theorem}
\autoref{main theorem: approximate parametric map} shows that, for parametric map approximation, the ReLU DNN's depth is $\mathcal{O}(\log_2(1/\varepsilon )\log_2(\log_2(1/\varepsilon)))$, and the number of non-zero parameters is $\mathcal{O}(d^3\log^2_2(1/\varepsilon))$. Our result theoretically guarantees that the DNN approximation can overcome the curse of dimensionality, as the complexity bound only depends on the reduced basis dimension $d$ rather than the high-fidelity dimension $N$. 

\autoref{main theorem: approximate parametric map} extends the Galerkin framework results, i.e., Theorem 4.3 of \cite{Kutyniok2022Theoreticala} and Theorem 5.4 of \cite{Lei2022Solving}, to the collocation framework. Furthermore, we adopt a different expansion form of the Neumann series in the proof and hence obtain a sharper bound on the number of non-zero parameters compared with $\mathcal{O}(d^3\log^2_2(1/\varepsilon)\log^2_2(\log_2(1/\varepsilon)))$ of \cite{Kutyniok2022Theoreticala}. Compared to the bounds we previously obtained in \cite{Lei2022Solving} for ReQU DNNs ($\mathcal{O}(\log_2(\log_2(1/\varepsilon)))$ for depth and $\mathcal{O}(d^3\log^2_2(\log_2(1/\varepsilon)))$ for number of non-zero parameters), the bounds we obtain for ReLU DNNs are looser. This is because ReQU DNNs can perfectly represent matrix multiplication with zero error, while ReLU DNNs only approximate it. See Proposition 4.3 of \cite{Lei2022Solving} and Proposition 3.7 of \cite{Kutyniok2022Theoreticala} in this paper. However, notice that ReLU DNNs are more widely used in practice, so \autoref{main theorem: approximate parametric map} still holds significant importance.

\section{Theoretic Analysis}\label{section: Theoretic Analysis}

In this section, we first present some basic estimates on neural network approximation in the form of lemmas. Subsequently, we will provide the proof of \autoref{main theorem: approximate parametric map}.

\subsection{Basic Definitions and Estimates}
We first introduce concatenation and parallelization operations for neural networks, which may be used to construct complex neural networks from simple networks. Furthermore, we derive the complexity estimates of operations of ReLU neural networks.
Next, we define the concatenation of two neural networks.
\begin{definition}\label{neural networkdefinition}
	Let $L_{1}, L_{2} \in \mathbb{N}$ and let 
	$\Phi^{1}=\left(\left(W_{1}^{1}, b_{1}^{1}\right), \ldots,\left(W_{L_{1}}^{1}, b_{L_{1}}^{1}\right)\right)$, $\Phi^{2}=\left(\left(W_{1}^{2}, b_{1}^{2}\right), \ldots\right.$, $\left.\left(W_{L_{2}}^{2}, b_{L_{2}}^{2}\right)\right)$ be two neural networks with activation function $\sigma$. The input dimension of $\Phi^{1}$ is the same as the output dimension of $\Phi^{2}$. Next, we denote $\Phi^{1} \circ \Phi^{2}$ as the concatenation of $\Phi^1$, $\Phi^2$ as follows:
	\[
		\begin{aligned}
			\Phi^{1} \circ \Phi^{2}:=\bigg(&\left(W_{1}^{2}, b_{1}^{2}\right), \ldots,\left(W_{L_{2}-1}^{2}, b_{L_{2}-1}^{2}\right), \\
			&\left(W_{1}^{1} W_{L_{2}}^{2}, W_{1}^{1} b_{L_{2}}^{2}+b_{1}^{1}\right),\left(W_{2}^{1}, b_{2}^{1}\right), \ldots,\left(W_{L_{1}}^{1}, b_{L_{1}}^{1}\right)\bigg),
		\end{aligned}
	\]
	which yields $L(\Phi^{1} \circ \Phi^{2}) = L_1 + L_2 - 1$.
\end{definition}
Next lemma shows that $M\left(\Phi^{1} \circ \Phi^{2}\right)$ can be estimated by $\max \left\{M\left(\Phi^{1}\right), M\left(\Phi^{2}\right)\right\}$ in some special cases.
\begin{lemma}\label{lemma: M of concatenation}\cite[Lemma A.1]{Kutyniok2022Theoreticala}
	Let $\Phi$ be a neural network with $L(\Phi)$ layers, $n_0$-dimensional input and $n_L$-dimensional output. If $W \in \mathbb{R}^{1 \times n_L}$, then, for all $i=1, \ldots, L(\Phi)$,
	\[
		M_{i}(((W, 0)) \circ \Phi) \leq M_{i}(\Phi).
	\]
	In particular, $M((W, 0) \circ \Phi) \leq M(\Phi) .$ Moreover, if $W \in \mathbb{R}^{n_0 \times n}$ such that, for every $i \leq n_0$ there is at most one non-zero element in the $i$-th row $W_{i, :}$, then, for all $i=1, \ldots, L(\Phi)$,
	\[
		M_{i}\left(\Phi \circ\left(\left(W, \mathbf{0}\right)\right)\right) \leq M_{i}(\Phi).
	\]
	In particular, $M\left(\Phi \circ\left(\left(W, \mathbf{0}\right)\right)\right) \leq M(\Phi)$.
\end{lemma}
However, there is no bound on $M\left(\Phi^{1} \circ \Phi^{2}\right)$ that is linear in $M\left(\Phi^{1}\right)$ and $M\left(\Phi^{2}\right)$ in general case. Hence, we introduce an alternative sparse concatenation to control the number of non-zero parameters. We need the following lemma to construct the ReLU neural network of the identity function.
\begin{lemma}\cite[Lemma 3.3]{Kutyniok2022Theoreticala}
	For any $n, L \in \mathbb{N}$, there exists a neural network $\Phi_{n, L}^{\mathbf{Id}}$ with input dimension $n$ and output dimension $n$ such that
	\[
		\begin{aligned}
			R\left(\Phi_{n, L}^{\mathbf{Id}}\right) &=\mathbf{I} \mathbf{d}_{\mathbb{R}^n}, \\
			M(\Phi_{n, L}^{\mathbf{Id}})&\leq 2nL,
		\end{aligned}
	\]
and all the non-zero parameters are $\{-1, 1\}$-valued.
\end{lemma}

\begin{definition}
Let $\Phi^1, \Phi^2$ be two neural networks such that the output dimension of $\Phi^2$ and the input dimension of $\Phi^1$ equal $n \in \mathbb{N}$. Then, the sparse concatenation of $\Phi^1$ and $\Phi^2$ is defined as
\[
	\Phi^1 \odot \Phi^2:=\Phi^1 \circ \Phi_{n, 1}^{\mathbf{I d}} \circ \Phi^2
\]
\end{definition}
We now introduce the parallelization operation.
\begin{definition}
	If $\Phi^{1}, \ldots, \Phi^{k}$ are neural networks that have equal input dimension and number of layers, such that $\Phi^{i}=\left(\left(W_{1}^{i}, b_{1}^{i}\right), \ldots,\left(W_{L}^{i}, b_{L}^{i}\right)\right)$. We define the parallelization of $\Phi^{1}, \ldots, \Phi^{k}$ with
	\[
		P\left(\Phi^{1}, \ldots, \Phi^{k}\right):= \left(\left(W_{1}, b_{1}\right), \ldots,\left(W_{L}, b_{L}\right)\right),
	\]
	where 
	\[
		W_i := 
		\begin{pmatrix}
			W_{i}^{1} & & \\
			& W_{i}^{2} & & \\
			&  & \ddots&\\
			& & & W_{i}^{k}
		\end{pmatrix}, \quad
		b_i := 
		\begin{pmatrix}
			b_{i}^{1} \\
			b_{i}^{2} \\
			\vdots \\
			b_{i}^{k}
		\end{pmatrix}, \quad 
		i = 1, \ldots, L.
	\]
	If $\Phi^{1}, \ldots, \Phi^{k}$ are neural networks that have equal input dimensions but different numbers of layers. Let
	\[
		L:=\max \left\{L\left(\Phi^{1}\right), \ldots, L\left(\Phi^{k}\right)\right\},
	\]
	and we define the extension 
	\[
		E_{L}(\Phi^{i}):=
		\left\{
		\begin{aligned}
			&\Phi^{i}, &&  L(\Phi^{i})=L, \\
			&\Phi_{n^i_{L(\Phi^{i})}, L-L(\Phi^{i})}^{\mathbf{Id}} \odot \Phi^{i}, &&  L(\Phi^{i})<L,
		\end{aligned}\right.
	\]		
	where $n^i_{L(\Phi^{i})}$ is the output dimension of $\Phi^{i}$. We define the parallelization of $\Phi^{1}, \ldots, \Phi^{k}$ with
	\[
		P\left(\Phi^{1}, \ldots, \Phi^{k}\right):=P\left(E_{L}\left(\Phi^{1}\right), \ldots, E_{L}\left(\Phi^{k}\right)\right).
	\]
\end{definition}
The following two lemmas provide the properties of the sparse concatenation and the parallelization of neural networks.
\begin{lemma}\label{lemma: sparse concatenation}\cite[Lemma 5.3]{Kutyniok2022Theoreticala}
	Let $\Phi^{1}, \Phi^{2}$ be ReLU neural networks. The input dimension of $\Phi^{1}$ equals the output dimension of $\Phi^{2}$. Then, for the sparse concatenation $\Phi^{1} \odot \Phi^{2}$ it holds
	\begin{enumerate}
		\item $R\left(\Phi^{1}\right) \left(R\left(\Phi^{2}\right)(\cdot)\right)= R\left(\Phi^{1} \odot \Phi^{2}\right)(\cdot)$,
		\item $L\left(\Phi^{1} \odot \Phi^{2}\right) = L\left(\Phi^{1}\right)+L\left(\Phi^{2}\right)$,
		\item $M\left(\Phi^1 \odot \Phi^2\right) \leq 2 M\left(\Phi^1\right)+ 2 M\left(\Phi^2\right)$.
	\end{enumerate}
\end{lemma}
\begin{lemma}\label{lemma: parallelization}\cite[Lemma 5.4]{Kutyniok2022Theoreticala}
	Let $\Phi^{1}, \ldots, \Phi^{k}$ be ReLU neural networks with equal input dimension $n$. Then, for the parallelization $P\left(\Phi^{1}, \ldots, \Phi^{k}\right)$ it holds
	\begin{enumerate}
		\item for all $x_{1}, \ldots x_{k} \in \mathbb{R}^{n}$, 
		\[
			R\left(P\left(\Phi^{1}, \ldots, \Phi^{k}\right)\right)\left(x_{1}, \ldots, x_{k}\right) = \left(R\left(\Phi^{1}\right)\left(x_{1}\right), \ldots, R\left(\Phi^{k}\right)\left(x_{k}\right)\right),
		\]
		\item $L\left(P\left(\Phi^{1}, \ldots, \Phi^{k}\right)\right) = \max \{L\left(\Phi^{1}\right), \ldots, L\left(\Phi^{k}\right)\}$,
		\item Let $n^i_{L(\Phi^{i})}$ denote the output dimension of $\Phi^{i}$, then 
		\[
			M\left(P\left(\Phi^1, \ldots, \Phi^k\right)\right) \leq 2\left(\sum_{i=1}^k M\left(\Phi^i\right)\right)+4\left(\sum_{i=1}^k n^i_{L(\Phi^{i})}\right) \max _{i=1, \ldots, k} L\left(\Phi^i\right),
		\]
		\item $M\left(P\left(\Phi^{1}, \ldots, \Phi^{k}\right)\right)=\sum_{i=1}^{k} M\left(\Phi^{i}\right)$, if $L\left(\Phi^{1}\right) = \cdots = L\left(\Phi^{k}\right).$
	\end{enumerate}
\end{lemma}
\subsection{Neural Network Approximation of the Parametric Map}

To approximate the parametric map $\mu \mapsto \mathbf{c}_\mu$ using deep ReLU neural networks, we tackle the problem in two steps: first, approximating the inverse of a matrix, and second, approximating matrix multiplication. Initially, we give the following lemmas for approximating matrix multiplication with ReLU neural networks. 
\begin{lemma}\label{mult_mnk}\cite[Proposition 3.7]{Kutyniok2022Theoreticala}
	Let $m, n, k \in \mathbb{N}, Z>0$, and $\varepsilon \in(0,1)$. There exists a ReLU neural network $\Phi_{\mathrm{mult} ; \varepsilon}^{Z, m, n, k}$ with $n(m+k)$-dimensional input, $mk$-dimensional output such that, for a universal constant $C_{\mathrm{mult}}>0$, we have
	\begin{enumerate}
		\item $\displaystyle \sup_{\substack{A \in \mathbb{R}^{m \times n}, B \in\mathbb{R}^{n \times k}, \\ \|A\|_2 \leq Z, \|B\|_2 \leq Z}}\left\|AB-\mathrm{matr}\left(R\left(\Phi_{\mathrm{mult}; \varepsilon}^{Z, m, n, k}\right)(\mathrm{vec}(A), \mathrm{vec}(B))\right)\right\|_2 \leq \varepsilon$,
		\item $L\left(\Phi_{\mathrm{mult}; \varepsilon}^{Z, m, n, k}\right) \leq C_{\mathrm{mult}} \cdot\left(\log_2(1 / \varepsilon)+\log_2(n \sqrt{mk})+\log_2(\max\{1, Z\})\right)$,
		\item $M\left(\Phi_{\mathrm{mult}; \varepsilon}^{Z, m, n, k}\right) \leq C_{\mathrm{mult}}m n k \cdot\left(\log_2(1 / \varepsilon)+\log_2(n \sqrt{mk})+\log_2(\max\{1, Z\})\right)$.
	\end{enumerate}
\end{lemma}
\begin{lemma}\label{power}\cite[Proposition A.3]{Kutyniok2022Theoreticala}
	Let $n,i \in \mathbb{N}, Z>0, \varepsilon \in(0,1/4)$ and $\delta \in (0,1)$. There exists a ReLU neural network $\Phi_{2^i ; \varepsilon}^{1-\delta, n}$ with $n^2$-dimensional input, $n^2$-dimensional output such that, for a universal constant $C_{\mathrm{sq}}>0$, we have
	\begin{enumerate}
		\item $\displaystyle \sup_{\substack{A \in \mathbb{R}^{n \times n}, \|A\|_2 \leq 1-\delta}}\left\|A^{2^{i}}-\mathrm{matr}\left(R\left(\Phi_{2^i ; \varepsilon}^{1-\delta, n}\right)(\mathrm{vec}(A))\right)\right\|_2 \leq \varepsilon$,
		\item $L\left(\Phi_{2^i ; \varepsilon}^{1-\delta, n}\right) \leq C_{\mathrm{sq}}i \cdot\left(\log_2(1 / \varepsilon)+\log_2 n+i\right)$,
		\item $M\left(\Phi_{2^i ; \varepsilon}^{1-\delta, n}\right) \leq C_{\mathrm{sq}}in^3 \cdot\left(\log_2(1 / \varepsilon)+\log_2 n + i\right)$.
	\end{enumerate}
\end{lemma}
For the inverse operator, we prove the following proposition. By leveraging a different expansion form of the Neumann series, we improve the bound of the number of non-zero parameters, compared with a similar result in \cite[Theorem 3.8]{Kutyniok2022Theoreticala}.
\begin{proposition}\label{prop: matrix inverse}
	For $\varepsilon, \delta \in (0,1)$, we define
	\[
		l = l(\varepsilon, \delta):=\left\lceil\log_2\left(\log_{1-\delta}\left(\frac{1}{2}\delta\varepsilon\right)+1\right)\right\rceil,
	\]
	There exists a universal constant $C_{\mathrm{inv}}>0$ such that for every $n \in \mathbb{N}, \varepsilon \in (0, 1/4)$ and $\delta \in (0, 1)$, there exists a ReLU neural network $\Phi_{\mathrm{inv};\varepsilon}^{1-\delta, n}$ with $n^{2}$-dimensional input and $n^{2}$-dimensional output satisfying the following properties:
	\begin{enumerate}
		\item $\displaystyle \sup_{A \in \mathbb{R}^{n \times n},\|A\|_2\leq 1-\delta}\left\|\left(I_{n}-A\right)^{-1}-\mathrm{matr}\left(R\left(\Phi_{\mathrm{inv};\varepsilon}^{1-\delta, n}\right)(\mathrm{vec}(A))\right)\right\|_{2} \leq \varepsilon,$
		\item $L\left(\Phi_{\mathrm{inv};\varepsilon}^{1-\delta, n}\right) \leq C_{\mathrm{inv}}l\cdot \left(\log_2(1 / \varepsilon)+ \log_2(1/\delta) + \log_2 n + l\right)$,
		\item $M\left(\Phi_{\mathrm{inv};\varepsilon}^{1-\delta, n}\right) \leq C_{\mathrm{inv}} 2^l n^2(n + l)\cdot \left(\log_2(1 / \varepsilon)+ \log_2(1/\delta) + \log_2 n+l\right).$
	\end{enumerate}
\end{proposition}
\begin{proof}
Based on the properties of the partial sums of the Neumann series, for every $\varepsilon, \delta \in(0,1)$ and $A \in \mathbb{R}^{n \times n}$ satisfying $\|A\|_2 \leq 1-\delta$, we have 
\[
	\left\|\left(I_n-A\right)^{-1}-\sum_{i=0}^{2^l-1} A^i\right\|_2 = \left\|\left(I_n-A\right)^{-1} A^{2^l}\right\|_2 \leq \frac{1}{1-(1-\delta)} \cdot(1-\delta)^{2^l} \leq \frac{\varepsilon}{2}.
\]
Thus, it suffices to construct a ReLU neural network representing the partial sums. Note that
\[
	\sum_{i=0}^{2^l-1} A^i=\prod_{i=0}^{l-1}\left(A^{2^i}+I_n\right).
\]
Let $Z_1:= 1-\delta, Z_2:= 3 + \varepsilon + 1/\delta$. For $i = 1, \ldots, l$, define $\Phi^1:= \left(\left(I_{n^2}, \mathrm{vec}(I_{n^2})\right)\right)$, 
\[
	\Pi^i:= 
	\begin{cases}
		\Phi^1, &  i=1, \\ 
		\Phi_{\mathrm{mult}; 5^{i-l-1}\varepsilon/2}^{Z_2, n, n, n} \odot P\left(\Pi^{i-1}, \Phi^1 \odot \Phi_{2^{i-1}; 5^{i-l-1}\delta\varepsilon/2}^{Z_1, n}\right), &  i \geq 2,
	\end{cases}
\]
and 
\[
	\Phi^i:=\Pi^i \circ \left(\left(
	\begin{pmatrix}
		I_{n^2} \\
		\vdots \\
		I_{n^2}
	\end{pmatrix}
	, \mathbf{0}\right)\right),
\]
where $I_{n^2}$ repeats $i$ times. We claim that, for $i = 1, \ldots, l$, we have
\[
	\left\|\mathrm{matr}\left(R\left(\Phi^i\right)(\mathrm{vec}(A))\right)-\sum_{k=0}^{2^i-1} A^k\right\|_2 \leq \frac{\varepsilon}{2\cdot 5^{l-i}}.
\]
We prove this by induction. The case $i = 1$ is trivial. If the claim holds for $i' = 1, \ldots, i-1$, then 
\[
	\begin{aligned}
		&\left\|\mathrm{matr}\left(R\left(\Phi^i\right)(\mathrm{vec}(A))\right)-\sum_{k=0}^{2^i-1} A^k\right\|_2 \\
		\leq{}&  \left\|\sum_{k=0}^{2^i-1} A^k - \left(\mathrm{matr}\left(R\left(\Phi_{2^{i-1}; 5^{i-l-1}\delta\varepsilon/2}^{Z_1, n}\right)(\mathrm{vec}(A))\right) + I_n\right)\prod_{k=0}^{i-2}\left(A^{2^k}+I_n\right)\right\|_2 \\
		&+ \left\|\left(\mathrm{matr}\left(R\left(\Phi_{2^{i-1}; 5^{i-l-1}\delta\varepsilon/2}^{Z_1, n}\right)(\mathrm{vec}(A))\right) + I_n\right)\prod_{k=0}^{i-2}\left(A^{2^k}+I_n\right)\right. \\ 
		&\quad - \left.\left(\mathrm{matr}\left(R\left(\Phi_{2^{i-1}; 5^{i-l-1}\delta\varepsilon/2}^{Z_1, n}\right)(\mathrm{vec}(A))\right) + I_n\right)\mathrm{matr}\left(R\left(\Phi^{i-1}\right)(\mathrm{vec}(A))\right)\right\|_2 \\
		&+ \left\|\left(\mathrm{matr}\left(R\left(\Phi_{2^{i-1}; 5^{i-l-1}\delta\varepsilon/2}^{Z_1, n}\right)(\mathrm{vec}(A))\right) + I_n\right)\mathrm{matr}\left(R\left(\Phi^{i-1}\right)(\mathrm{vec}(A))\right)\right. \\
		&\quad - \left.\mathrm{matr}\left(R\left(\Phi^i\right)(\mathrm{vec}(A))\right)\right\|_2 \\
		:={}& \mathbf{I} + \mathbf{II} + \mathbf{III}
	\end{aligned}
\]
For term $\mathbf{I}$, by \autoref{power} we have 
\[
	\mathbf{I} \leq \left\|\sum_{k=0}^{2^{i-1}-1}A^k\right\|_2 \cdot 5^{i-l-1}\delta\varepsilon/2 \leq 5^{i-l-1}\varepsilon/2.
\]
For term $\mathbf{II}$, by induction we have 
\[
	\mathbf{II} \leq \left\|\mathrm{matr}\left(R\left(\Phi_{2^{i-1}; 5^{i-l-1}\delta\varepsilon/2}^{Z_1, n}\right)(\mathrm{vec}(A))\right) + I_n\right\|_2  \cdot 5^{i-l-1}\varepsilon/2 \leq 3 \cdot 5^{i-l-1}\varepsilon/2.
\]
For term $\mathbf{III}$, by \autoref{mult_mnk} we have $\mathbf{III} \leq 5^{i-l-1}\varepsilon/2$. Combining $\mathbf{I}, \mathbf{II}$ and $\mathbf{III}$, we have proven the claim. Let $\Phi_{\mathrm{inv};\varepsilon}^{1-\delta, n}:= \Phi^l$, then this ReLU neural network satisfies the error estimate of the proposition.

Next, we analyze the size of the resulting neural network $\Phi^l$ by induction. For all $i = 1, \ldots, l$, by \autoref{lemma: sparse concatenation}, \autoref{lemma: parallelization}, \autoref{mult_mnk} and \autoref{power}, we have 
\[
	\begin{aligned}
		&L(\Phi^i) = L(\Pi^i) \\
		\leq{}& L(\Phi_{\mathrm{mult}; 5^{i-l-1}\varepsilon/2}^{Z_2, n, n, n}) + \max\{L(\Pi^{i-1}), L(\Phi_{2^{i-1}; 5^{i-l-1}\delta\varepsilon/2}^{Z_1, n}) + 1\} \\
		\leq{}&C_{\mathrm{mult}} \cdot \left(\log_2(1 / \varepsilon)+ l + \log_2 n+\log_2(\max\{1, Z_2\})\right) \\
		&+ \max\{C_{\mathrm{sq}}(i-1) \cdot\left(\log_2(1 / \varepsilon)+ \log_2(1/\delta) + \log_2 n+l\right), L(\Pi^{i-1})\}.
	\end{aligned}
\]
Let $C_{0}:= 2C_{\mathrm{mult}} + C_{\mathrm{sq}}$. Note that $L(\Pi^{1}) \leq C_{0} \cdot\left(\log_2(1 / \varepsilon)+ \log_2(1/\delta) + \log_2 n+l\right)$. By induction, we have 
\[
	L(\Phi^i) \leq C_{0}i\cdot \left(\log_2(1 / \varepsilon)+ \log_2(1/\delta) + \log_2 n+l\right).
\]
By \autoref{lemma: M of concatenation}, \autoref{lemma: sparse concatenation}, \autoref{lemma: parallelization}, \autoref{mult_mnk} and \autoref{power}, we have 
\[
	\begin{aligned}
		&M(\Phi^i) \leq M(\Pi^i) \\
		\leq{}& 2M(\Phi_{\mathrm{mult}; 5^{i-l-1}\varepsilon/2}^{Z_2, n, n, n}) + 2M(\Pi^{i-1}) + 4M(\Phi^1) + 4M(\Phi_{2^{i-1}; 5^{i-l-1}\delta\varepsilon/2}^{Z_1, n}) \\
		&+ 8n^2\max\{L(\Pi^{i-1}), L(\Phi_{2^{i-1}; 5^{i-l-1}\delta\varepsilon/2}^{Z_1, n}) + 1\}\\
		\leq{}& 2M(\Pi^{i-1}) + 2C_{\mathrm{mult}}n^3\cdot\left(\log_2(1 / \varepsilon)+ l + \log_2 n +\log_2(\max\{1, Z_2\})\right) \\
		&+4C_{\mathrm{sq}}(i-1)n^3 \cdot\left(\log_2(1 / \varepsilon)+\log_2(1 / \delta) + \log_2 n + l\right) \\
		&+8C_{0}(i-1)n^2\cdot \left(\log_2(1 / \varepsilon)+ \log_2(1/\delta) + \log_2 n+l\right).
	\end{aligned}
\]
Let $C_1:= 4C_{\mathrm{sq}} + 8C_0$. By induction, we have 
\[
	\begin{aligned}
		M(\Phi^i) \leq{}& 2^{i}n^2 + 2^i C_{\mathrm{mult}}n^3\cdot\left(\log_2(1 / \varepsilon)+ l + \log_2 n + \log_2 (1/\delta)\right)\\
		&+ 2^i i C_1 n^2\cdot \left(\log_2(1 / \varepsilon)+ \log_2(1/\delta) + \log_2 n+l\right).
	\end{aligned}
\]
With a suitably chosen $C_{\mathrm{inv}} > 0$, we have completed the proof.
\end{proof}

Recall that the coefficient vector under the reduced basis is calculated by
\[
	\mathbf{c}_{\mu} = \left(B_{\mu}^T B_{\mu}\right)^{-1} B_{\mu}^T
	\begin{bmatrix}
        \mathbf{f}_\mu \\
        \mathbf{g}_\mu
    \end{bmatrix},
\]
where $B_{\mu} \in \mathbb{R}^{N \times d}$ is defined by \eqref{B_mu}. We now approximate the map $\mu \mapsto \mathbf{c}_{\mu}$ using ReLU neural networks. We have already shown that ReLU neural networks can approximate the inverse and multiplication operators. It remains to assume that the maps $\mu \mapsto (\mathbf{f}^T_\mu, \mathbf{g}^T_\mu)^T, \mu \mapsto B_{\mu}$ can also be approximated by ReLU neural networks, which is exactly fulfilled by \autoref{assumption: mu to B} and \autoref{assumption: mu to fg}. \autoref{assumption: spectrum of BTB} is related to a scaling factor. By \autoref{prop: matrix inverse}, we can apply $\Phi_{\mathrm{inv};\varepsilon}^{1-\delta, d}$ on $\mathrm{vec}(I_{d} - B_{\mu}^T B_{\mu})$ to approximate $(B_{\mu}^T B_{\mu})^{-1}$ if $\|I_{d} - B_{\mu}^T B_{\mu}\|_2 \leq 1-\delta$. To guarantee that, we introduce a scaling factor $\lambda>0$ and ensure $\|I_{d} - \lambda B_{\mu}^T B_{\mu}\|_2 \leq 1-\delta$. Under \autoref{assumption: spectrum of BTB}, we can select $\lambda:=\left(\alpha^2 + \beta^2\right)^{-1}$ and $\delta:=\lambda \beta^2$, then $\|I_{d} - \lambda(B_{\mu}^T B_{\mu})\|_{2} \leq 1-\delta$ for all $\mu \in \mathcal{D}$ and possible reduced basis $V$. We fix values of $\lambda$ and $\delta$ for the remainder of this paper. Now, we present the construction of the neural network to approximate $\mu \mapsto (B_{\mu}^T B_{\mu})^{-1}$.
\begin{proposition}\label{prop: inv BTB}
	Suppose that \autoref{assumption: spectrum of BTB} and \autoref{assumption: mu to B} hold. For any $\varepsilon \in (0,1/4)$ and reduced basis $V \in \mathbb{R}^{N \times d}$, there exists a ReLU neural network $\Phi_{\mathrm{inv};\varepsilon}^{B^TB}$ with $p$-dimensional input and $d^2$-dimensional output that satisfies the following properties.
	\begin{enumerate}
		\item $\displaystyle \sup_{\mu \in \mathcal{D}}\left\|(B_{\mu}^T B_{\mu})^{-1} - \mathrm{matr}\left(R\left(\Phi_{\mathrm{inv};\varepsilon}^{B^TB}\right)(\mu)\right)\right\|_{2} \leq \varepsilon.$
		\item There exists a constant $C^{B}_{L}$ only depending on $\alpha$ and $\beta$, such that 
		\[
			\begin{aligned}
				L(\Phi_{\mathrm{inv};\varepsilon}^{B^TB}) \leq{}& C^{B}_{L} \cdot \log_2(\log_2(1/\varepsilon))\cdot(\log_2(1/\varepsilon) + \log_2 d) \\ 
				& + C^{B}_{L}\log_2 N + L(B; V, \varepsilon_1),
			\end{aligned}
		\]
		\item There exists a constant $C^{B}_{M}$ only depending on $\alpha$ and $\beta$, such that 
		\[
			\begin{aligned}
				&M(\Phi_{\mathrm{inv};\varepsilon}^{B^TB}) \\
				\leq{}& C^{B}_{M}d^2 \log_2(1/\varepsilon) \cdot (\log_2(\log_2(1/\varepsilon)) + d) \cdot(\log_2(1/\varepsilon) + \log_2 d) \\ 
				&+C^{B}_{M} N d^2 \cdot \left(\log_2(1/\varepsilon) + \log_2 N \right) + 8M(B; V, \varepsilon_1).
			\end{aligned}
		\]
	\end{enumerate}
	Here, $\displaystyle \varepsilon_1 := \min\left\{\frac{\varepsilon \beta^4}{24\alpha}, \frac{\beta^2\sqrt{\varepsilon}}{4}, \frac{\beta^2}{12\alpha}, \frac{\beta}{3}\right\}.$
\end{proposition}
\begin{proof}
	Let $\varepsilon_1, \varepsilon_2 > 0$ be determined later. We define
	\[
		\Phi_{\varepsilon_1, \varepsilon_2}^{B^TB} := \Phi_{\mathrm{mult}; \varepsilon_2}^{\alpha + \varepsilon_1, d, N, d} \odot P\left(\Phi_{V, \varepsilon_1}^{B^T}, \Phi_{V, \varepsilon_1}^{B}\right) \circ \left(\left(
		\begin{pmatrix}
			\mathbf{I}_{p} \\
			\mathbf{I}_{p}
		\end{pmatrix}, \mathbf{0} \right)\right).
	\]
	By the triangle inequality, we have 
	\[
		\begin{aligned}
			&\sup_{\mu \in \mathcal{D}}\left\|B_{\mu}^T B_{\mu} - \mathrm{matr}\left(R\left(\Phi_{\varepsilon_1, \varepsilon_2}^{B^TB}\right)(\mu)\right)\right\|_{2} \\
			\leq{}&\sup_{\mu \in \mathcal{D}}\left\|B_{\mu}^T B_{\mu} - B_{\mu}^T \mathrm{matr}\left(R\left(\Phi_{V, \varepsilon_1}^{B}\right)(\mu)\right)\right\|_{2} \\
			&+ \sup_{\mu \in \mathcal{D}}\left\|B_{\mu}^T \mathrm{matr}\left(R\left(\Phi_{V, \varepsilon_1}^{B}\right)(\mu)\right) - \mathrm{matr}\left(R\left(\Phi_{V, \varepsilon_1}^{B^T}\right)(\mu)\right) \mathrm{matr}\left(R\left(\Phi_{V, \varepsilon_1}^{B}\right)(\mu)\right)\right\|_{2} \\
			&+ \sup_{\mu \in \mathcal{D}}\left\| \mathrm{matr}\left(R\left(\Phi_{V, \varepsilon_1}^{B^T}\right)(\mu)\right) \mathrm{matr}\left(R\left(\Phi_{V, \varepsilon_1}^{B}\right)(\mu)\right) - \mathrm{matr}\left(R\left(\Phi_{\varepsilon_1, \varepsilon_2}^{B^TB}\right)(\mu)\right)\right\|_{2} \\
			=:{}& \mathbf{I} + \mathbf{II} + \mathbf{III}.
		\end{aligned}
	\]
	For term $\mathbf{I}$, by \autoref{assumption: spectrum of BTB} and \autoref{assumption: mu to B} we have
	\[
		\mathbf{I} \leq \sup_{\mu \in \mathcal{D}}\|B^T_{\mu}\|_2 \cdot \left\|B_{\mu} - \mathrm{matr}\left(R\left(\Phi_{V, \varepsilon_1}^{B}\right)(\mu)\right)\right\|_{2} \leq \alpha \varepsilon_1.
	\]
	For term $\mathbf{II}$, by \autoref{assumption: mu to B} we have
	\[
		\begin{aligned}
			\mathbf{II} \leq \sup_{\mu \in \mathcal{D}}\left\| B_{\mu}^T  - \mathrm{matr}\left(R\left(\Phi_{V, \varepsilon_1}^{B^T}\right)(\mu)\right) \right\|_{2} \cdot \| \mathrm{matr}\left(R\left(\Phi_{V, \varepsilon_1}^{B}\right)(\mu)\right) \|_2 \leq (\alpha+\varepsilon_1) \varepsilon_1.
		\end{aligned}
	\]
	For term $\mathbf{III}$, note that $\left\|\mathrm{matr}\left(R\left(\Phi_{V, \varepsilon_1}^{B}\right)(\mu)\right)\right\|_2$ and $\displaystyle \left\|\mathrm{matr}\left(R\left(\Phi_{V, \varepsilon_1}^{B^T}\right)(\mu)\right)\right\|_2$ are below $(\alpha + \varepsilon_1)$. Hence, we can use \autoref{mult_mnk}, which derives $\mathbf{III} \leq \varepsilon_2$. Then we can calculate the size of $\Phi_{\varepsilon_1, \varepsilon_2}^{B^TB}$. By \autoref{lemma: sparse concatenation}, \autoref{lemma: parallelization}, \autoref{mult_mnk},  and \autoref{assumption: mu to B}, we have 
	\[
		\begin{aligned}
			&L(\Phi_{\varepsilon_1, \varepsilon_2}^{B^TB}) \\
			\leq{}& L\left(\Phi_{\mathrm{mult}; \varepsilon_2}^{\alpha + \varepsilon_1, d, N, d}\right) + \max\bigg\{L\left(\Phi_{V, \varepsilon_1}^{B^T}\right), L\left(\Phi_{V, \varepsilon_1}^{B}\right)\bigg\} \\
			\leq{}& C_{\mathrm{mult}} \cdot\left(\log_2(1 / \varepsilon_2)+\log_2(Nd)+\log_2(\max\{1, \alpha + \varepsilon_1\})\right) + L(B; V, \varepsilon_1).
		\end{aligned}
	\]
	By \autoref{lemma: M of concatenation}, \autoref{lemma: sparse concatenation}, \autoref{lemma: parallelization}, \autoref{mult_mnk},  and \autoref{assumption: mu to B}, we have 
	\[
		\begin{aligned}
			&M(\Phi_{\varepsilon_1, \varepsilon_2}^{B^TB}) \\
			\leq{}& 2M\left(\Phi_{\mathrm{mult}; \varepsilon_2}^{\alpha + \varepsilon_1, d, N, d}\right) + 4M(B; V, \varepsilon_1) \\
			\leq{}& 2C_{\mathrm{mult}} Nd^2 \cdot \left(\log_2(1 / \varepsilon_2)+\log_2(Nd)+\log_2(\max\{1, \alpha + \varepsilon_1\})\right) \\
			&+ 4M(B; V, \varepsilon_1).
		\end{aligned}
	\]
	Now assume that $\Phi_{\varepsilon_1, \varepsilon_2}^{B^TB} = \left(\left(W_1, b_1\right), \ldots,\left(W_L, b_L\right)\right)$, we construct another ReLU neural network as 
	\[
		\Phi_{\varepsilon_1, \varepsilon_2}^{I - \lambda B^TB} := \left(\left(W_1, b_1\right), \ldots,\left(-\lambda W_L, -\lambda b_L + \mathrm{vec}(I_{d})\right)\right).
	\]
    We have 
	\begin{equation}\label{Term I II III}
		\begin{aligned}
			&\sup_{\mu \in \mathcal{D}}\left\|I_{d} - \lambda B_{\mu}^T B_{\mu} - \mathrm{matr}\left(R\left(\Phi_{\varepsilon_1, \varepsilon_2}^{I - \lambda B^TB}\right)(\mu)\right)\right\|_{2} \\
			={}&\sup_{\mu \in \mathcal{D}}\left\|I_{d} - \lambda B_{\mu}^T B_{\mu} - \left(I_{d} - \lambda\mathrm{matr}\left(R\left(\Phi_{\varepsilon_1, \varepsilon_2}^{B^TB}\right)(\mu)\right)\right)\right\|_{2} \\
			\leq{}& \lambda(\mathbf{I} + \mathbf{II} + \mathbf{III}) \\
			\leq{}& \lambda(2\alpha\varepsilon_1 + \varepsilon_1^2 + \varepsilon_2),
		\end{aligned}
	\end{equation}
    as well as 
	\begin{equation}\label{size of B^TB, I}
		\begin{aligned}
			L(\Phi_{\varepsilon_1, \varepsilon_2}^{I - \lambda B^TB}) &= L(\Phi_{\varepsilon_1, \varepsilon_2}^{B^TB}), \\
			M(\Phi_{\varepsilon_1, \varepsilon_2}^{I - \lambda B^TB}) &\leq M(\Phi_{\varepsilon_1, \varepsilon_2}^{B^TB}) + d.
		\end{aligned}
	\end{equation}

	As the final stage, we define the ReLU neural network
	\[
		\Phi_{\mathrm{inv} ; \varepsilon_1, \varepsilon_2, \varepsilon_3}^{B^TB} := \left(\left(\lambda I_{d}, \mathbf{0} \right)\right) \circ \Phi_{\mathrm{inv} ; \varepsilon_3}^{1 - \delta/2,d} \odot \Phi_{\varepsilon_1, \varepsilon_2}^{I - \lambda B^TB},
	\]
	with $p$-dimensional input and $d^2$-dimensional output. The value $\varepsilon_3 > 0$ is undetermined. To bound the approximation error, we estimate 
	\begin{equation}
		\begin{aligned}
			&\sup_{\mu \in \mathcal{D}}\left\|(B_{\mu}^T B_{\mu})^{-1} - \mathrm{matr}\left(R\left(\Phi_{\mathrm{inv} ; \varepsilon_1, \varepsilon_2, \varepsilon_3}^{B^TB}\right)(\mu)\right)\right\|_{2} \\
			\leq{}&\sup_{\mu \in \mathcal{D}}\left\|(B_{\mu}^T B_{\mu})^{-1} - \lambda \left(I_{d} - \mathrm{matr}\left(R\left(\Phi_{\varepsilon_1, \varepsilon_2}^{I - \lambda B^TB}\right)(\mu)\right)\right)^{-1}\right\|_{2} \\
			&+\sup_{\mu \in \mathcal{D}}\left\|\lambda \left(I_{d} - \mathrm{matr}\left(R\left(\Phi_{\varepsilon_1, \varepsilon_2}^{I - \lambda B^TB}\right)(\mu)\right)\right)^{-1} \right.\\
			&\left. \qquad \qquad - \lambda \mathrm{matr}\left(R\left( \Phi_{\mathrm{inv} ; \varepsilon_3}^{1 - \delta/2,d} \odot \Phi_{\varepsilon_1, \varepsilon_2}^{I - \lambda B^TB} \right)(\mu)\right)\right\|_{2} \\
			=:{}& \mathbf{IV} + \mathbf{V}.
		\end{aligned}
	\end{equation}
	For term $\mathbf{IV}$, using the equation $C^{-1} - D^{-1} = C^{-1}(D - C)D^{-1}$, we write 
	\[
		\begin{aligned}
			\mathbf{IV} &\leq \sup_{\mu \in \mathcal{D}}\left\|(B_{\mu}^T B_{\mu})^{-1}\right\|_{2} \cdot \left\|\lambda^{-1}\left(I_{d} - \mathrm{matr}\left(R\left(\Phi_{\varepsilon_1, \varepsilon_2}^{I - \lambda B^TB}\right)(\mu)\right)\right) - B_{\mu}^T B_{\mu}\right\|_{2}\\
			&\qquad \qquad \cdot \left\|\lambda \left(I_{d} - \mathrm{matr}\left(R\left(\Phi_{\varepsilon_1, \varepsilon_2}^{I - \lambda B^TB}\right)(\mu)\right)\right)^{-1}\right\|_{2}. 
		\end{aligned}
	\]
	\autoref{assumption: spectrum of BTB} and \eqref{Term I II III} yields that 
	\[
		\sup_{\mu \in \mathcal{D}} \left\|\lambda^{-1}\left(I_{d} - \mathrm{matr}\left(R\left(\Phi_{\varepsilon_1, \varepsilon_2}^{I - \lambda B^TB}\right)(\mu)\right)\right) - B_{\mu}^T B_{\mu}\right\|_{2} \leq 2\alpha\varepsilon_1 + \varepsilon_1^2 + \varepsilon_2, 
	\]
	and
	\[
		\begin{aligned}
			&\sup_{\mu \in \mathcal{D}} \left\|\lambda \left(I_{d} - \mathrm{matr}\left(R\left(\Phi_{\varepsilon_1, \varepsilon_2}^{I - \lambda B^TB}\right)(\mu)\right)\right)^{-1}\right\|_{2} \\
			\leq{}& \lambda / \left(\inf_{\mu \in \mathcal{D}}\left\|\lambda B_{\mu}^T B_{\mu}\right\|_{2} - \lambda(2\alpha\varepsilon_1 + \varepsilon_1^2 + \varepsilon_2)\right) \\
			\leq{}& (\beta^2 -(2\alpha\varepsilon_1 + \varepsilon_1^2 + \varepsilon_2))^{-1}.
		\end{aligned}
	\]
	Here, we shall make sure that $\beta^2 -(2\alpha\varepsilon_1 + \varepsilon_1^2 + \varepsilon_2) > 0$ when determine the values of $\varepsilon_1, \varepsilon_2$. Then term $\mathbf{IV}$ can be bounded as 
	\[
		\mathbf{IV} \leq \beta^{-2}(2\alpha\varepsilon_1 + \varepsilon_1^2 + \varepsilon_2)(\beta^2 -(2\alpha\varepsilon_1 + \varepsilon_1^2 + \varepsilon_2))^{-1}.
	\]
	For term $\mathbf{V}$, \eqref{Term I II III} yields that 
	\[
		\begin{aligned}
			\sup_{\mu \in \mathcal{D}}\left\|\mathrm{matr}\left(R\left(\Phi_{\varepsilon_1, \varepsilon_2}^{I - \lambda B^TB} \right)(\mu)\right)\right\|_{2} 
			&\leq \lambda(2\alpha\varepsilon_1 + \varepsilon_1^2 + \varepsilon_2) + \sup_{\mu \in \mathcal{D}}\left\|I_{d} - \lambda B_{\mu}^T B_{\mu}\right\|_{2} \\
			&\leq \lambda(2\alpha\varepsilon_1 + \varepsilon_1^2 + \varepsilon_2) + 1 - \lambda\beta^2.
		\end{aligned}
	\]
	Assume that $\lambda(2\alpha\varepsilon_1 + \varepsilon_1^2 + \varepsilon_2) + 1 - \lambda\beta^2 \leq 1 - \delta/2$, then by \autoref{prop: matrix inverse}, $\mathbf{V} \leq \lambda\varepsilon_3.$ Combining $\mathbf{IV}$ and $\mathbf{V}$ implies
	\begin{equation}\label{Term IV V}
		\begin{aligned}
			&\sup_{\mu \in \mathcal{D}}\left\|(B_{\mu}^T B_{\mu})^{-1} - \mathrm{matr}\left(R\left(\Phi_{\mathrm{inv} ; \varepsilon_1, \varepsilon_2, \varepsilon_3}^{B^TB}\right)(\mu)\right)\right\|_{2} \\
			\leq{}& \mathbf{IV} + \mathbf{V} \leq \beta^{-2}(2\alpha\varepsilon_1 + \varepsilon_1^2 + \varepsilon_2)(\beta^2 -(2\alpha\varepsilon_1 + \varepsilon_1^2 + \varepsilon_2))^{-1} + \lambda\varepsilon_3.
		\end{aligned}
	\end{equation}
	For the size of neural network, by \autoref{lemma: sparse concatenation}, \autoref{prop: matrix inverse} and \eqref{size of B^TB, I}, we have
	\begin{align*}
		&L(\Phi_{\mathrm{inv} ; \varepsilon_1, \varepsilon_2, \varepsilon_3}^{B^TB}) \\ 
		={}& L(\Phi_{\mathrm{inv} ; \varepsilon_3}^{1 - \delta/2,d}) + L(\Phi_{\varepsilon_1, \varepsilon_2}^{B^TB}) \\
		\leq{}& C_{\mathrm{inv}} l(\varepsilon_3,\delta/2)\left(\log_2 (1 / \varepsilon_3)+l(\varepsilon_3,\delta/2)+ \log_2(1/\delta) + \log_2(d)\right) \\
		&+ C_{\mathrm{mult}} \cdot\left(\log_2(1 / \varepsilon_2)+\log_2(Nd)+\log_2(\max\{1, \alpha + \varepsilon_1\})\right) + L(B; V, \varepsilon_1)
	\end{align*}
	and
	\begin{align*}
		&M(\Phi_{\mathrm{inv} ; \varepsilon_1, \varepsilon_2, \varepsilon_3}^{B^TB}) \\ 
		\leq{}& 2M(\Phi_{\mathrm{inv} ; \varepsilon_3}^{1 - \delta/2,d}) + 2M(\Phi_{\varepsilon_1, \varepsilon_2}^{B^TB}) + 2d \\
		\leq{}& 2C_{\mathrm{inv}} 2^{l(\varepsilon_3,\delta/2)} d^2 (d + l(\varepsilon_3,\delta/2)) \\
		& \qquad \qquad \cdot \left(\log_2 (1/\varepsilon_3)+l(\varepsilon_3,\delta/2)+ \log_2(1/\delta)+ \log_2 (d)\right) \\
		&+ 4C_{\mathrm{mult}} Nd^2 \cdot \left(\log_2(1 / \varepsilon_2)+\log_2(Nd)+\log_2(\max\{1, \alpha + \varepsilon_1\})\right)  \\
		&+ 8M(B; V, \varepsilon_1) + 2d.
	\end{align*}
	It remains to determine the values of $\varepsilon_1, \varepsilon_2, \varepsilon_3$. Recall that we should make sure  
	\[
		\beta^2 -(2\alpha\varepsilon_1 + \varepsilon_1^2 + \varepsilon_2) > 0, \quad \lambda(2\alpha\varepsilon_1 + \varepsilon_1^2 + \varepsilon_2) + 1 - \lambda\beta^2 \leq 1 - \delta/2.
	\]
	Besides, the approximation error bound in the conclusion should hold, which can be proved by \eqref{Term IV V} if 
	\[
		\beta^{-2}(2\alpha\varepsilon_1 + \varepsilon_1^2 + \varepsilon_2)(\beta^2 -(2\alpha\varepsilon_1 + \varepsilon_1^2 + \varepsilon_2))^{-1} + \lambda\varepsilon_3 \leq \varepsilon.
	\]
	One can check that the following values are sufficient. 
	\[
		\varepsilon_1 = \min\left\{\frac{\varepsilon \beta^4}{24\alpha}, \frac{\beta^2\sqrt{\varepsilon}}{4}, \frac{\beta^2}{12\alpha}, \frac{\beta}{3}\right\}, \quad \varepsilon_2 = \min\left\{\frac{\beta^4 \varepsilon}{12}, \frac{\beta^2}{6}\right\},  \quad \varepsilon_3 = \frac{\varepsilon}{2\lambda}.
	\]
	Let $\Phi_{\mathrm{inv};\varepsilon}^{B^TB}:=\Phi_{\mathrm{inv} ; \varepsilon_1, \varepsilon_2, \varepsilon_3}^{B^TB}$ and plug in these values yields
	\[	
		\sup_{\mu \in \mathcal{D}}\left\|(B_{\mu}^T B_{\mu})^{-1} - \mathrm{matr}\left(R\left(\Phi_{\mathrm{inv};\varepsilon}^{B^TB}\right)(\mu)\right)\right\|_{2} \leq \frac{\varepsilon}{2} + \frac{\varepsilon}{2} \leq \varepsilon,
	\]
	\[
		L(\Phi_{\mathrm{inv};\varepsilon}^{B^TB}) \leq C^{B}_{L} \cdot \log_2(\log_2(1/\varepsilon))\cdot(\log_2(1/\varepsilon) + \log_2 d) + C^{B}_{L}\log_2 N + L(B; V, \varepsilon_1),
	\]
	\[
		\begin{aligned}
			M(\Phi_{\mathrm{inv};\varepsilon}^{B^TB}) \leq{}& C^{B}_{M}d^2 \log_2(1/\varepsilon) \cdot (\log_2(\log_2(1/\varepsilon)) + d) \cdot(\log_2(1/\varepsilon) + \log_2 d) \\ 
			&+C^{B}_{M} N d^2 \cdot \left(\log_2(1/\varepsilon) + \log_2 N \right) + 8M(B; V, \varepsilon_1),
		\end{aligned}
	\]
	where $C^{B}_{L}$ and $C^{B}_{M}$ are suitably chosen constants only depending on $\alpha$ and $\beta$. This completes the proof. 
\end{proof}
We can now construct ReLU neural networks that approximate the parametric map $\mu \mapsto \mathbf{c}_\mu$ and hence complete the proof of \autoref{main theorem: approximate parametric map}
\begin{proof}[Proof of \autoref{main theorem: approximate parametric map}]
	Let $\varepsilon_1, \varepsilon_2, \varepsilon_3, \varepsilon_4, \varepsilon_5, Z_1, Z_2 > 0$ be determined later. Let $\Phi_{V, \varepsilon_1}^{B^T}$ be the neural network defined in \autoref{assumption: mu to B} and $\Phi_{\varepsilon_2}^{\mathbf{fg}}$ be the neural network defined in \autoref{assumption: mu to fg}. We define 
	\[
		\Phi_{\varepsilon_1, \varepsilon_2, \varepsilon_3}^{B^T\mathbf{fg}} := \Phi_{\mathrm{mult}; \varepsilon_3}^{Z_1, d, N, 1} \odot P\left(\Phi_{V, \varepsilon_1}^{B^T}, \Phi_{\varepsilon_2}^{\mathbf{fg}}\right) \circ \left(\left(
			\begin{pmatrix}
				I_p \\
				I_p
			\end{pmatrix}, 
			\mathbf{0}\right)\right).
	\]
	Let $\Phi_{\mathrm{inv};\varepsilon_4}^{B^TB}$ be the neural network derived in \autoref{prop: inv BTB}. We define 
	\[
		\Phi_{\mathrm{para}; \varepsilon_{1:5}}^{\mathbf{c}} := \Phi_{\mathrm{mult}; \varepsilon_5}^{Z_2, d, d, 1} \odot P\left(\Phi_{\mathrm{inv};\varepsilon_4}^{B^TB}, \Phi_{\varepsilon_1, \varepsilon_2, \varepsilon_3}^{B^T\mathbf{fg}}\right) \circ \left(\left(
		\begin{pmatrix}
			I_p \\
			I_p
		\end{pmatrix}, 
		\mathbf{0}\right)\right),
	\]	
	According to the triangle inequality, we obtain the estimation
	\[
		\begin{aligned}
			&\sup_{\mu \in \mathcal{D}}\left\|\mathbf{c}_\mu - R\left(\Phi_{\mathrm{para}; \varepsilon_{1:5}}^{\mathbf{c}}\right)(\mu)\right\|_{2} \\
			\leq{}& \sup_{\mu \in \mathcal{D}}\left\|(B_{\mu}^T B_{\mu})^{-1}B_{\mu}^T(\mathbf{f}^T_\mu,\mathbf{g}^T_\mu)^T - (B_{\mu}^T B_{\mu})^{-1}R\left(\Phi_{\varepsilon_1, \varepsilon_2, \varepsilon_3}^{B^T\mathbf{fg}}\right)(\mu)\right\|_{2} \\
			&+\sup_{\mu \in \mathcal{D}}\left\|(B_{\mu}^T B_{\mu})^{-1}R\left(\Phi_{\varepsilon_1, \varepsilon_2, \varepsilon_3}^{B^T\mathbf{fg}}\right)(\mu) - \mathrm{matr}\left(R\left(\Phi_{\mathrm{inv};\varepsilon_4}^{B^TB}\right)(\mu)\right)R\left(\Phi_{\varepsilon_1, \varepsilon_2, \varepsilon_3}^{B^T\mathbf{fg}}\right)(\mu)\right\|_{2} \\
			&+\sup_{\mu \in \mathcal{D}}\left\|\mathrm{matr}\left(R\left(\Phi_{\mathrm{inv};\varepsilon_4}^{B^TB}\right)(\mu)\right)R\left(\Phi_{\varepsilon_1, \varepsilon_2, \varepsilon_3}^{B^T\mathbf{fg}}\right)(\mu) - R\left(\Phi_{\mathrm{para}; \varepsilon_{1:5}}^{\mathbf{c}}\right)(\mu)\right\|_{2} \\
			=:{}& \mathbf{I} + \mathbf{II} + \mathbf{III}.
		\end{aligned}
	\]
	The term $\mathbf{I}$ can be further bounded by 
	\[
		\begin{aligned}
			\mathbf{I} \leq{}& \sup_{\mu \in \mathcal{D}}\left\|(B_{\mu}^T B_{\mu})^{-1}\right\|_2 \cdot \left(\left\|B_{\mu}^T(\mathbf{f}^T_\mu,\mathbf{g}^T_\mu)^T - B_{\mu}^TR\left(\Phi_{\varepsilon_2}^{\mathbf{fg}}\right)(\mu)\right\|_{2} \right.\\ 
			&+ \left\|B_{\mu}^TR\left(\Phi_{\varepsilon_2}^{\mathbf{fg}}\right)(\mu) - \mathrm{matr}\left(R\left(\Phi_{V,\varepsilon_1}^{B^T}\right)(\mu)\right)R\left(\Phi_{\varepsilon_2}^{\mathbf{fg}}\right)(\mu)\right\|_{2}  \\
			&+ \left. \left\|\mathrm{matr}\left(R\left(\Phi_{V,\varepsilon_1}^{B^T}\right)(\mu)\right)R\left(\Phi_{\varepsilon_2}^{\mathbf{fg}}\right)(\mu) - R\left(\Phi_{\varepsilon_1, \varepsilon_2, \varepsilon_3}^{B^T\mathbf{fg}}\right)(\mu)\right\|_{2} \right) \\
			\leq{}& \beta^{-2}(\mathbf{IV} + \mathbf{V} + \mathbf{VI}).
		\end{aligned}
	\]

	For term $\mathbf{IV}$ and $\mathbf{V}$, \autoref{assumption: spectrum of BTB}, \autoref{assumption: mu to B} and \autoref{assumption: mu to fg} yield $\mathbf{IV} \leq \alpha \varepsilon_2, \mathbf{V} \leq (\gamma + \varepsilon_2) \varepsilon_1$.

	For term $\mathbf{VI}$, we assume that 
	\[
		\begin{aligned}
			\sup_{\mu \in \mathcal{D}}\left\|\mathrm{matr}\left(R\left(\Phi_{V,\varepsilon_1}^{B^T}\right)(\mu)\right)\right\|_2 &\leq \alpha + \varepsilon_1 \leq Z_1, \\
			\sup_{\mu \in \mathcal{D}}\left\|R\left(\Phi_{\varepsilon_2}^{\mathbf{fg}}\right)(\mu)\right\|_2 &\leq \gamma + \varepsilon_2 \leq Z_1.
		\end{aligned}
	\]
	Then, by \autoref{mult_mnk} we have $\mathbf{VI} \leq \varepsilon_3.$

	For term $\mathbf{I}$, combining the estimates above yields $\mathbf{I} \leq \beta^{-2}(\alpha \varepsilon_2 + (\gamma + \varepsilon_2) \varepsilon_1 + \varepsilon_3).$
	
	For term $\mathbf{II}$, by the bound of term $\mathbf{IV}, \mathbf{V}$ and $\mathbf{VI}$ we have 
	\[
		\begin{aligned}
			\sup_{\mu \in \mathcal{D}}\left\|R\left(\Phi_{\varepsilon_1, \varepsilon_2, \varepsilon_3}^{B^T\mathbf{fg}}\right)(\mu)\right\|_2 &\leq \mathbf{IV} + \mathbf{V} + \mathbf{VI} + \left\|B_{\mu}^T(\mathbf{f}^T_\mu,\mathbf{g}^T_\mu)^T\right\|_2 \\
			&\leq \alpha \varepsilon_2 + (\gamma + \varepsilon_2) \varepsilon_1 + \varepsilon_3 + \alpha \gamma.
		\end{aligned}
	\]
	Hence, by \autoref{prop: inv BTB} we have 
	\[
		\begin{aligned}
			\mathbf{II} &\leq \sup_{\mu \in \mathcal{D}}\left\|(B_{\mu}^T B_{\mu})^{-1} - \mathrm{matr}\left(R\left(\Phi_{\mathrm{inv};\varepsilon_4}^{B^TB}\right)(\mu)\right)\right\|_{2} \cdot \left\|R\left(\Phi_{\varepsilon_1, \varepsilon_2, \varepsilon_3}^{B^T\mathbf{fg}}\right)(\mu)\right\|_{2} \\
			&\leq \varepsilon_4 (\alpha \varepsilon_2 + (\gamma + \varepsilon_2) \varepsilon_1 + \varepsilon_3 + \alpha \gamma).
		\end{aligned}
	\]

	For term $\mathbf{III}$, we assume that 
	\[
		\begin{aligned}
			\sup_{\mu \in \mathcal{D}}\left\|\mathrm{matr}\left(R\left(\Phi_{\mathrm{inv};\varepsilon_4}^{B^TB}\right)(\mu)\right)\right\|_2 &\leq \varepsilon_4 + \beta^{-2} \leq Z_2, \\
			\sup_{\mu \in \mathcal{D}}\left\|R\left(\Phi_{\varepsilon_1, \varepsilon_2, \varepsilon_3}^{B^T\mathbf{fg}}\right)(\mu)\right\|_2 &\leq \alpha \varepsilon_2 + (\gamma + \varepsilon_2) \varepsilon_1 + \varepsilon_3 + \alpha \gamma \leq Z_2.
		\end{aligned}
	\]
	Then, by \autoref{mult_mnk} we have $\mathbf{III} \leq \varepsilon_5.$

	Combining these estimates yields
	\[
		\begin{aligned}
			&\sup_{\mu \in \mathcal{D}}\left\|\mathbf{c}_\mu - R\left(\Phi_{\mathrm{para}; \varepsilon_{1:5}}^{\mathbf{c}}\right)(\mu)\right\|_{2} \leq \mathbf{I} + \mathbf{II} + \mathbf{III} \\
			\leq{}&\beta^{-2}(\alpha \varepsilon_2 + (\gamma + \varepsilon_2) \varepsilon_1 + \varepsilon_3) + \varepsilon_4 (\alpha \varepsilon_2 + (\gamma + \varepsilon_2) \varepsilon_1 + \varepsilon_3 + \alpha \gamma) + \varepsilon_5.
		\end{aligned}
	\]
	Next, we estimate the size of the ReLU neural networks $\Phi_{\mathrm{para}; \varepsilon_{1:5}}^{\mathbf{c}}$. By \autoref{lemma: sparse concatenation}, \autoref{lemma: parallelization}, \autoref{mult_mnk} and \autoref{prop: inv BTB}, we have 
	\[
		\begin{aligned}
			&L(\Phi_{\mathrm{para}; \varepsilon_{1:5}}^{\mathbf{c}}) \\
			\leq{}& L(\Phi_{\mathrm{mult}; \varepsilon_5}^{Z_2, d, d, 1}) \\
			&+ \max\left\{L(\Phi_{\mathrm{inv};\varepsilon_4}^{B^TB}), L(\Phi_{\mathrm{mult}; \varepsilon_3}^{Z_1, d, N, 1}) + \max\left\{L(\Phi_{V, \varepsilon_1}^{B^T}), L(\Phi_{\varepsilon_2}^{\mathbf{fg}})\right\}\right\} \\
			\leq{}& C_{\mathrm{mult}} \cdot\left(\log_2(1 / \varepsilon_5)+\log_2(d \sqrt{d})+\log_2(\max\{1, Z_2\})\right) \\
			&+ C^{B}_{L} \cdot \log_2(\log_2(1/\varepsilon_4))\cdot(\log_2(1/\varepsilon_4) + \log_2 d) \\ 
			&+ C^{B}_{L}\log_2 N + L(B; V, \varepsilon_{4,1}) \\
			&+ C_{\mathrm{mult}} \cdot\left(\log_2(1 / \varepsilon_3)+\log_2(N \sqrt{d})+\log_2(\max\{1, Z_1\})\right) \\
			&+ L(B; V, \varepsilon_1) + L(\mathbf{fg}; \varepsilon_2),
		\end{aligned}
	\]
	where $\displaystyle \varepsilon_{4,1}:= \min\left\{\frac{\beta^4 \varepsilon_4}{24\alpha}, \frac{\beta^2}{12\alpha}, \frac{\sqrt{\varepsilon_4}\beta^2}{4}, \frac{\beta}{3}\right\}$, and 
	\[
		\begin{aligned}
			&M(\Phi_{\mathrm{para}; \varepsilon_{1:5}}^{\mathbf{c}}) \\
			\leq{}& 2M(\Phi_{\mathrm{mult}; \varepsilon_5}^{Z_2, d, d, 1}) + 4M(\Phi_{\mathrm{inv};\varepsilon_4}^{B^TB}) \\
			&+ 8M(\Phi_{\mathrm{mult}; \varepsilon_3}^{Z_1, d, N, 1}) + 16M(\Phi_{V, \varepsilon_1}^{B^T}) + 16M(\Phi_{\varepsilon_2}^{\mathbf{fg}}) \\
			&+ 32(Nd + N)\max\left\{L(\Phi_{V, \varepsilon_1}^{B^T}), L(\Phi_{\varepsilon_2}^{\mathbf{fg}})\right\} \\
			&+ 8(d^2 + d)\max\left\{L(\Phi_{\mathrm{inv};\varepsilon_4}^{B^TB}), L(\Phi_{\mathrm{mult}; \varepsilon_3}^{Z_1, d, N, 1}) + \max\left\{L(\Phi_{V, \varepsilon_1}^{B^T}), L(\Phi_{\varepsilon_2}^{\mathbf{fg}})\right\}\right\} \\
			\leq{}& 2C_{\mathrm{mult}}d^2 \cdot \left(\log_2(1 / \varepsilon_5)+\log_2(d \sqrt{d})+\log_2(\max\{1, Z_2\})\right)\\
			&+ 4C^{B}_{M}d^2 \log_2(1/\varepsilon_4) \cdot (\log_2(\log_2(1/\varepsilon_4)) + d)\cdot(\log_2(1/\varepsilon_4) + \log_2 d) \\
			&+ 4C^{B}_{M} N d^2 \cdot  \left(\log_2(1/\varepsilon_4) + \log_2 N\right) + 8M(B; V, \varepsilon_{4,1}) \\
			&+ 8C_{\mathrm{mult}} N d \cdot\left(\log_2(1 / \varepsilon_3)+\log_2(N \sqrt{d})+\log_2(\max\{1, Z_1\})\right) \\
			&+ 16M(B; V, \varepsilon_1) + 16M(\mathbf{fg}; \varepsilon_2) \\
			&+ 64NdL(B; V, \varepsilon_1) + 64NdL(\mathbf{fg}; \varepsilon_2) + 16d^2L(B; V, \varepsilon_{4, 1}) \\
			&+ 16d^2 C^{B}_{L}\log_2(\log_2(1/\varepsilon_4))\cdot(\log_2(1/\varepsilon_4) + \log_2 d) + 16d^2 C^{B}_{L}\log_2 N  \\
			&+ 16d^2 C_{\mathrm{mult}} \cdot \left(\log_2(1 / \varepsilon_3)+\log_2(N \sqrt{d})+\log_2(\max\{1, Z_1\})\right) \\
			&+ 16d^2L(B; V, \varepsilon_1) + 16d^2L(\mathbf{fg}; \varepsilon_2).
		\end{aligned}
	\]

	It remains to determine the values of $\varepsilon_1, \ldots, \varepsilon_5$. We have to guarantee that 
	\[
		\begin{aligned}
			&\alpha + \varepsilon_1 \leq Z_1, \quad \gamma + \varepsilon_2 \leq Z_1, \quad \varepsilon_4 + \beta^{-2} \leq Z_2, \quad \alpha \varepsilon_2 + (\gamma + \varepsilon_2) \varepsilon_1 + \varepsilon_3 + \alpha \gamma \leq Z_2, \\
			&\beta^{-2}(\alpha \varepsilon_2 + (\gamma + \varepsilon_2) \varepsilon_1 + \varepsilon_3) + \varepsilon_4 (\alpha \varepsilon_2 + (\gamma + \varepsilon_2) \varepsilon_1 + \varepsilon_3 + \alpha \gamma) + \varepsilon_5 \leq \varepsilon,
		\end{aligned}
	\]
	which can be verified if 
	\[
		\begin{aligned}
			& \varepsilon_1 = \min \left\{\frac{\varepsilon\beta^2}{12\gamma}, \frac{\beta \sqrt{\varepsilon}}{4}, \frac{\alpha}{4}, \frac{\sqrt{\alpha \gamma}}{2}\right\}, \quad \varepsilon_2 = \min \left\{\frac{\varepsilon\beta^2}{12\alpha}, \frac{\beta \sqrt{\varepsilon}}{4}, \frac{\gamma}{4}, \frac{\sqrt{\alpha \gamma}}{2}\right\}, \\
			& \varepsilon_3 = \min \left\{\frac{\varepsilon\beta^2}{12}, \frac{\alpha \gamma}{4}\right\}, \quad \varepsilon_4 = \frac{\varepsilon}{6\alpha\gamma}, \quad \varepsilon_5 = \frac{\varepsilon}{3}, \\
			& Z_1 = \alpha + \gamma + \varepsilon_1 + \varepsilon_2, \quad Z_2 = \varepsilon_4 + \beta^{-2} + \alpha \varepsilon_2 + (\gamma + \varepsilon_2) \varepsilon_1 + \varepsilon_3 + \alpha \gamma.
		\end{aligned}
	\]
	Let $\Phi_{\mathrm{para};\varepsilon}^{\mathbf{c}} := \Phi_{\mathrm{para}; \varepsilon_{1:5}}^{\mathbf{c}}$ and plug in these values yields
	\[	
		\sup_{\mu \in \mathcal{D}}\left\|\mathbf{c}_\mu - \mathrm{matr}\left(R\left(\Phi_{\mathrm{para}; \varepsilon_{1:5}}^{\mathbf{c}}\right)(\mu)\right)\right\|_{2} \leq \frac{\varepsilon}{3} + \frac{\varepsilon}{3} + \frac{\varepsilon}{3} \leq \varepsilon,
	\]
	\[
		\begin{aligned}
			L(\Phi_{\mathrm{para};\varepsilon}^{\mathbf{c}}) \leq{}& C^{\mathbf{c}}_{L}\cdot\log_2(\log_2(1/\varepsilon))\cdot(\log_2(1/\varepsilon) +  \log_2 d) \\
			&+C^{\mathbf{c}}_{L} \log_2 N + L(B; V, \varepsilon_{4, 1}) + L(B; V, \varepsilon_1) + L(\mathbf{fg}; \varepsilon_2),
		\end{aligned}
	\]
	\[
		\begin{aligned}
			&M(\Phi_{\mathrm{para}; \varepsilon_{1:5}}^{\mathbf{c}}) \\
			\leq{}& C^{\mathbf{c}}_{M}N d^2 \cdot (\log_2(1/\varepsilon) + \log_2 N) \\
			&+ C^{\mathbf{c}}_{M}d^2 \log_2(1/\varepsilon) \cdot (\log_2(\log_2(1/\varepsilon)) + d)\cdot(\log_2(1/\varepsilon) + \log_2 d) \\
			&+ C^{\mathbf{c}}_{M}d^2  \cdot \log_2(\log_2(1/\varepsilon)) \cdot \log_2 N \\
			&+ 8M(B; V, \varepsilon_{4,1}) + 16M(B; V, \varepsilon_1) + 16M(\mathbf{fg}; \varepsilon_2) \\
			&+80NdL(B; V, \varepsilon_1) + 80NdL(\mathbf{fg}; \varepsilon_2) + 16d^2L(B; V, \varepsilon_{4, 1}),
		\end{aligned}
	\]
	where $C^{\mathbf{c}}_{L}$ and $C^{\mathbf{c}}_{M}$ are suitably chosen constants only depending on $\alpha, \beta$ and $\gamma$. We relabel $\varepsilon_{4, 1}$ as $\varepsilon_3$ in the main theorem, and the proof is completed. 
\end{proof}

\section{Numerical Experiments}\label{section: Numerical Experiments}
In this section, we verify our analysis with numerical examples. Our algorithm operates in two distinct phases: an offline phase and an online phase. In the offline phase, we leverage the POD method to produce the reduced basis and integrate it in the training of neural networks, which require the bulk of the computational workload. This setup enables the online phase to simultaneously and efficiently simulate solutions for a vast array of parameters, a feat achieved through the synergistic exploitation of neural network parallelism and matrix computations. Our algorithm demonstrates superior efficacy compared to the Least Squares Reduced Radial Basis Function Method (LS-RRBFM) described by \cite{Chen2016Reduced}, which employs a similar offline-online decomposition strategy. We also compare the efficiency of neural networks with or without the use of POD.

Specifically, in the offline phase, we generate the data set and train the neural network with this data set. We first solve the PPDE with a collocation method, which generally discretizes the domain $\Omega$ to get $N$ nodes in the interior and on the boundary $\partial \Omega$. In this paper, we consider the RBF-FD method, which is briefly introduced in \autoref{appendix section: RBF-FD Method for PDE}. The RBF-FD method is mesh-free and, therefore, suitable for irregular geometries. However, the accuracy of this method is heavily influenced by the discretization method. Generally, a set of nodes that allows low interpolation error cannot have a large region in $\Omega$ without centers, which means having a small fill distance. Many practical methods can produce well-distributed node sets. Readers may refer to recent progress \cite{Fornberg2015Fast, Shankar2018Robust, Suchde2023Pointa} that generate robust node sets suitable for RBF-FD. After the discretization, the RBF-FD method is ready to calculate a high-fidelity ground truth solution, or snapshot $\mathbf{u}_\mu$ given any $\mu \in \mathcal{D}$. By sampling the parameter space of $n_s$ parameter samples, we calculate the snapshot matrix $S \in \mathbb{R}^{N \times n_s}$. 

Then we employ the POD on $S$ and obtain the reduced basis matrix $V \in \mathbb{R}^{N \times d}$. To determine the reduced basis dimension $d$, we choose a desired error tolerance $\varepsilon_{\mathrm{POD}} > 0$ and calculate the POD dimension $d$ by 
\[
    d := \min\bigg\{n \in \mathbb{N} \mid \frac{\sum_{i=1}^n \sigma_i^2}{\sum_{i=1}^r \sigma_i^2} \geq 1-\varepsilon_{\mathrm{POD}}^2 \bigg\}
\]
where the rank $r \leq \min \left(N, n_s\right)$. Proposition 6.1. of \cite{Quarteroni2016Reduced} shows that the POD basis guarantees a minimal projection-recovery error among the selected snapshots, and this error also equals the square sum of the tail of the singular values, which must be less than $\varepsilon_{\mathrm{POD}}^2$. 

We also need to consider the values of $n_s, \varepsilon_{\mathrm{POD}}$, and the location of the parameter samples $\{\mu^1, \ldots, \mu^{n_s}\}$. The choices of $n_s, \varepsilon_{\mathrm{POD}}$ must balance the accuracy and computation time. The location of the samples should be well-distributed and representative. One may consider a uniform tensorial sampling with different grid step sizes in different dimensions, while other types of sparse grids are also possible. Different sizes and locations of $\{\mu^1, \ldots, \mu^{n_s}\}$ do not necessarily have a significant impact on $S$ and $d$. This depends on the nature of the equation. The key is to ensure that the problem is suitable for the use of the RCM and is reducible with the POD method. To verify that, we can check if the decay of the singular values of $S$ is sufficiently fast. If the rates of decay are essentially the same for different values of $n_s$ and locations of samples, a smaller $n_s$ and simpler sampling method will be more favored. See a more detailed discussion in \cite{Quarteroni2016Reduced}.

Now, we are ready to generate the data set. For any $\mu \in \mathcal{D}$, we solve the PPDE with RBF-FD and obtain the coefficient vector $V^T\mathbf{u}_{\mu} \in \mathbb{R}^{d}$ under the POD basis, which represents a parametric map $\mathcal{D} \to \mathbb{R}^{d}$. Hence, we can sample $n_{\mathrm{data}}$ points in $\mathcal{D}$ and calculate the input-output pairs $\{(\mu^i, V^T\mathbf{u}_{\mu^i})\}_{i=1}^{n_{\mathrm{data}}}$. We then train the neural networks using these data to approximate this parametric map. We note that, compared with the numerical algorithm presented in \cite{Geist2021Numerical} and \cite{Lei2022Solving}, the use of POD in our algorithm or other possible reduced basis methods, can greatly reduce the output dimension of the neural network from $N$ to $d$, which is beneficial for networks training convergence, reduction in network sizes and training time. We will compare the numerical performance of networks with and without POD.

In the online phase, predictions can be made in parallel if multiple parameters $\{\mu^{1}, \ldots, \mu^{n}\}$ are given simultaneously. We only need to compute one forward propagation through the neural network, denoted by $\mathbf{NN}$, and multiply the POD basis, which gives the prediction $V\mathbf{NN}([\mu^{1} \vert \cdots \vert \mu^{n}]) \in \mathbb{R}^{N \times n}$. We outline the entire algorithm we use in numerical experiments in \autoref{POD-DNN algorithm}.
\begin{algorithm}
	\caption{POD-DNN algorithm for parametric PDEs}\label{POD-DNN algorithm}
	\textbf{The offline phase}
	\begin{algorithmic}[1]
		\STATE Discretize the domain $\Omega$ to get $N$ nodes, and discretize the parameter space $\mathcal{D}$ by $\{\mu^1, \ldots, \mu^{n_s}\}$.
		\FOR{$i = 1, \ldots, n_s$}
		\STATE Use RBF-FD to solve the PPDE \eqref{PPDE} with parameter $\mu^i$, obtain high-fidelity solution $\mathbf{u}_{\mu^i}$.
		\ENDFOR	
		\STATE Concatenate the snapshot matrix $S = [\mathbf{u}_{\mu^1}, \ldots ,\mathbf{u}_{\mu^{n_s}}]$.
		\STATE Perform SVD on $S$ to compute the POD basis matrix $V$ with tolerance $\varepsilon_{\mathrm{POD}}$.
		\STATE Sample another $n_{\mathrm{data}}$ points $\{\mu^1, \ldots, \mu^{n_{\mathrm{data}}}\}$ in $\mathcal{D}$.
		\FOR{$i = 1, \ldots, n_{\mathrm{data}}$}
		\STATE Use RBF-FD to solve the PDE \eqref{PPDE} with parameter $\mu^i$, obtain high-fidelity solution $\mathbf{u}_{\mu^i}$ and coefficient vector $V^T\mathbf{u}_{\mu^i}$.
		\ENDFOR	
		\STATE Save the input-output pairs $\{(\mu^i, V^T\mathbf{u}_{\mu^i})\}_{i=1}^{n_{\mathrm{data}}}$ as data set.
		\STATE Train the network $\mathbf{NN}$ with the data set.
	\end{algorithmic}
	
	\textbf{The online phase}
	\begin{algorithmic}[1]
		\STATE Given multiple parameters $\{\mu^{1} , \ldots , \mu^{n}\}$, compute $V\mathbf{NN}([\mu^{1} \vert \cdots \vert \mu^{n}])$.
	\end{algorithmic}
\end{algorithm}

Now we delve into three PDE examples: 2-D Helmholtz equation, (2+1)-D sine-Gordon equation, and 3-D static Schrödinger equation. Our examples rigorously evaluate the algorithm across a comprehensive range of scenarios, including irregular geometries, time-dependent PDEs, nonlinear PDEs, and infinite-dimensional functional parameters. Most settings and selections of hyperparameters of RBF-FD, POD, and NN are summarized in \autoref{appendix section: Tables of Hyperparameters}.

\subsection{2-D Helmholtz equation}

We first consider the two-dimensional Helmholtz equation in (28a) of \cite{Chen2016Reduced}. 
\begin{equation}\label{two-dimensional Helmholtz equation}
    \begin{aligned}
        -u_{xx} - \mu_1 u_{yy} - \mu_2 u & = -10\sin(8x(y-1)), && (x, y) \in \Omega, \mu \in \mathcal{D} = [0.1, 4] \times [0, 2],\\
        u & = 0, && (x, y) \in \partial \Omega,
    \end{aligned}
\end{equation}
where $\Omega$, illustrated by \autoref{fig: Helmholtz discretization}, is given by the polar coordinates equation $r(\theta) < 0.8 + 0.1(\sin(6\theta) + \sin(3\theta)).$ The discretization shown in \autoref{fig: Helmholtz discretization} is generated by the algorithm proposed in \cite{Fornberg2015Fast} and a corresponding open-source code \cite{Mishra2019NodeLab}. These discrete nodes are slightly clustering toward the boundary.
\begin{figure}[htbp!]
    \centering
    \includegraphics[width=0.3\textwidth]{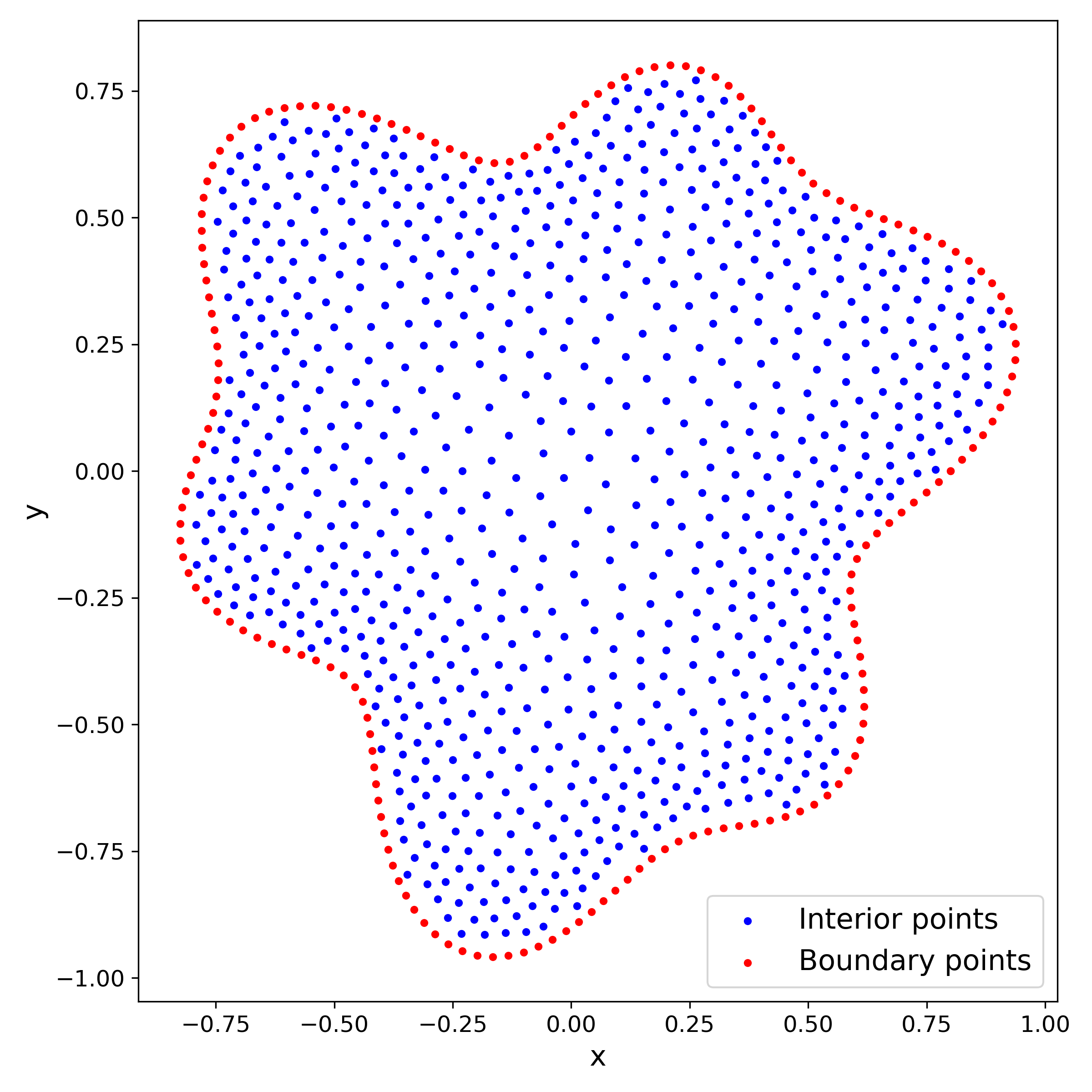}
    \caption{The discretization of Helmholtz equation.}
    \label{fig: Helmholtz discretization}
\end{figure}

As we have mentioned above, we can draw the spectral decay of the snapshot matrix $S$ to verify the reducibility of this equation under the POD method. \autoref{fig: Helmholtz spectrum} illustrates the spectral decay of $S$, and the proportion of energy (i.e., the square sum of singulars) discarded under the corresponding truncation position. We confirm that in this problem, the decay speed of singular values is sufficiently fast, which provides rationality for the implementation of POD.
\begin{figure}[htbp!]
    \centering
    \begin{subfigure}[b]{0.4\textwidth}
        \includegraphics[width=\linewidth]{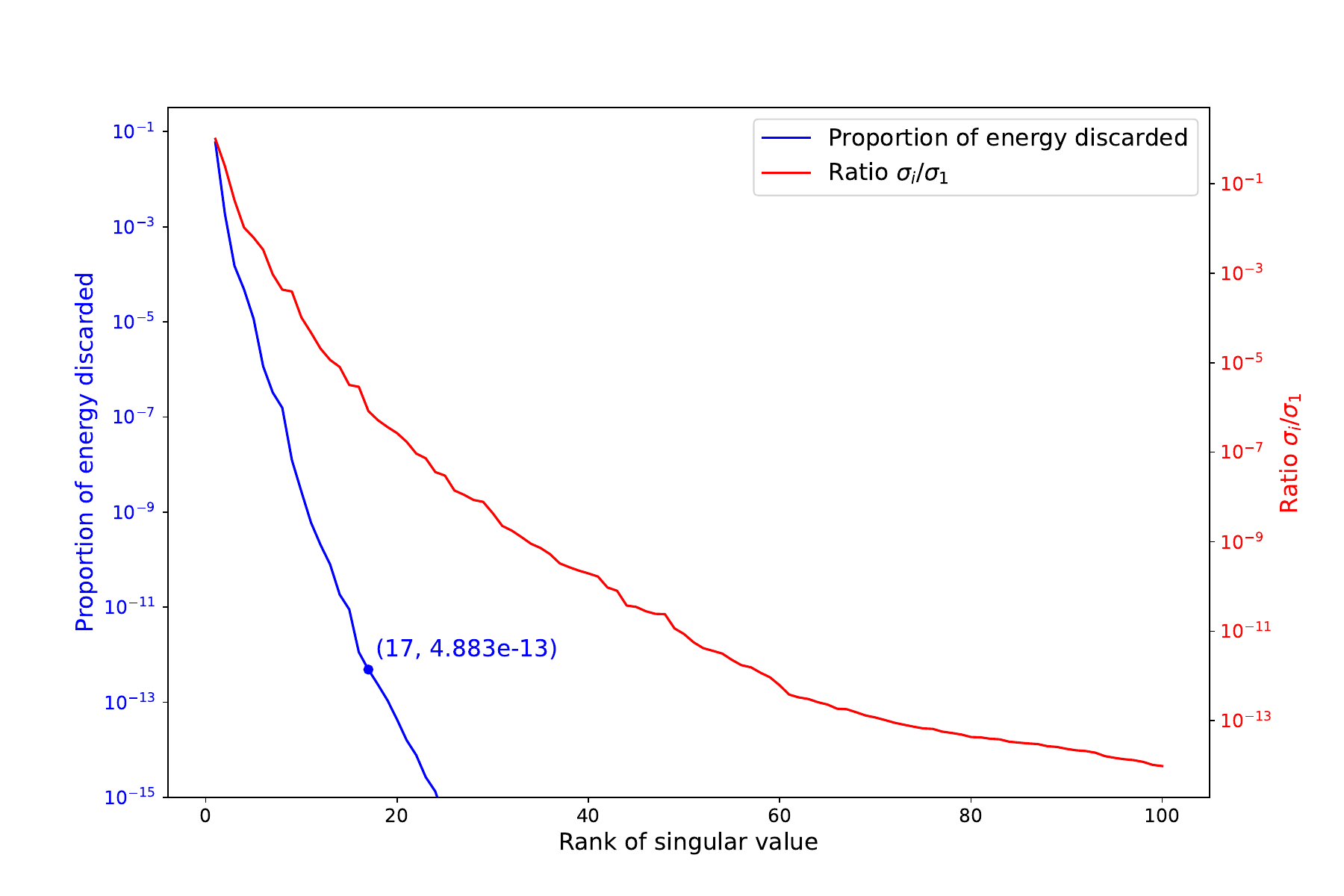}
        \caption{Helmholtz equation}
        \label{fig: Helmholtz spectrum}
    \end{subfigure}
    \begin{subfigure}[b]{0.4\textwidth}
        \includegraphics[width=\linewidth]{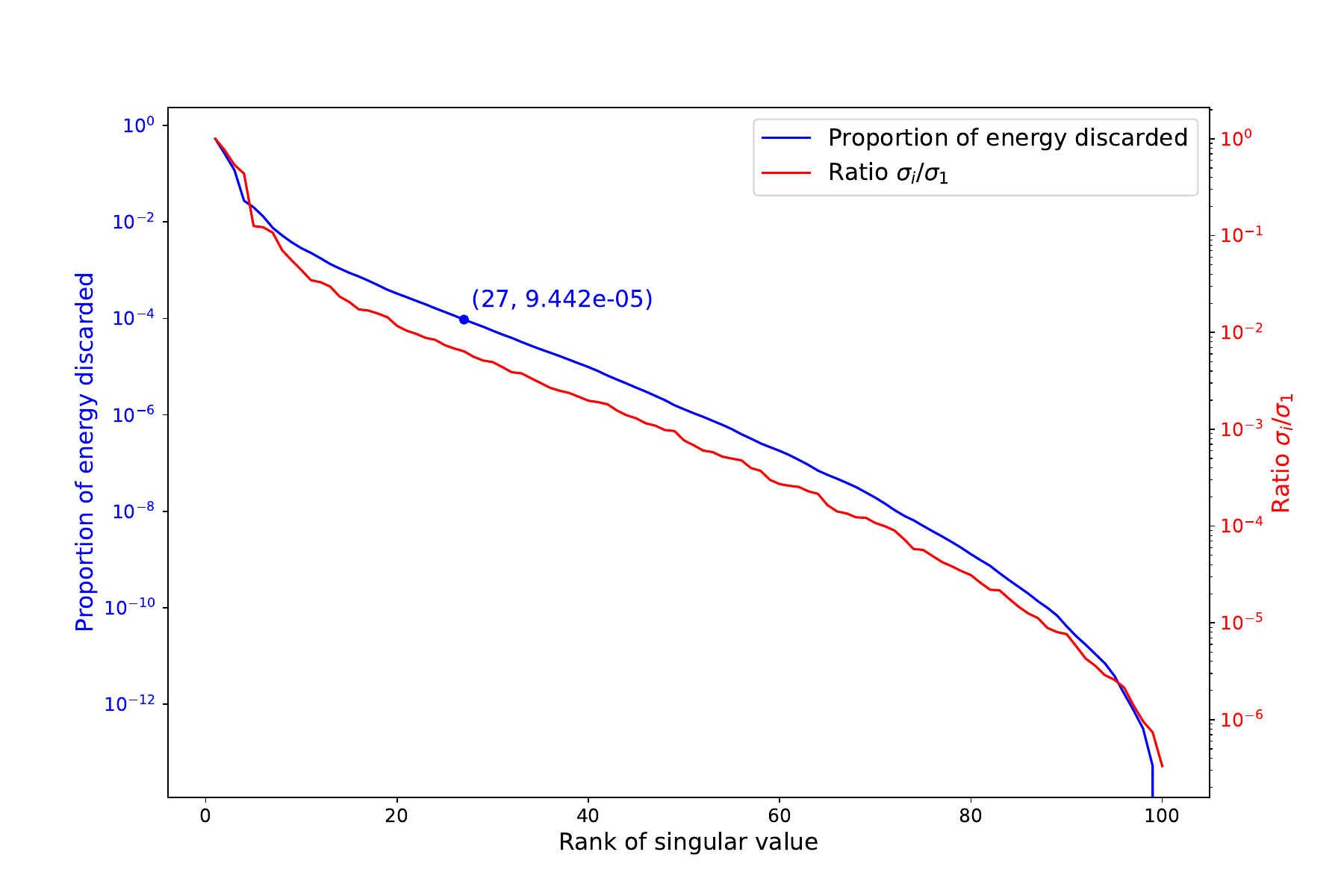}
        \caption{Sine-Gordon equation}
        \label{fig: sine-Gordon spectrum}
    \end{subfigure}

    \begin{subfigure}[b]{0.4\textwidth}
        \includegraphics[width=\linewidth]{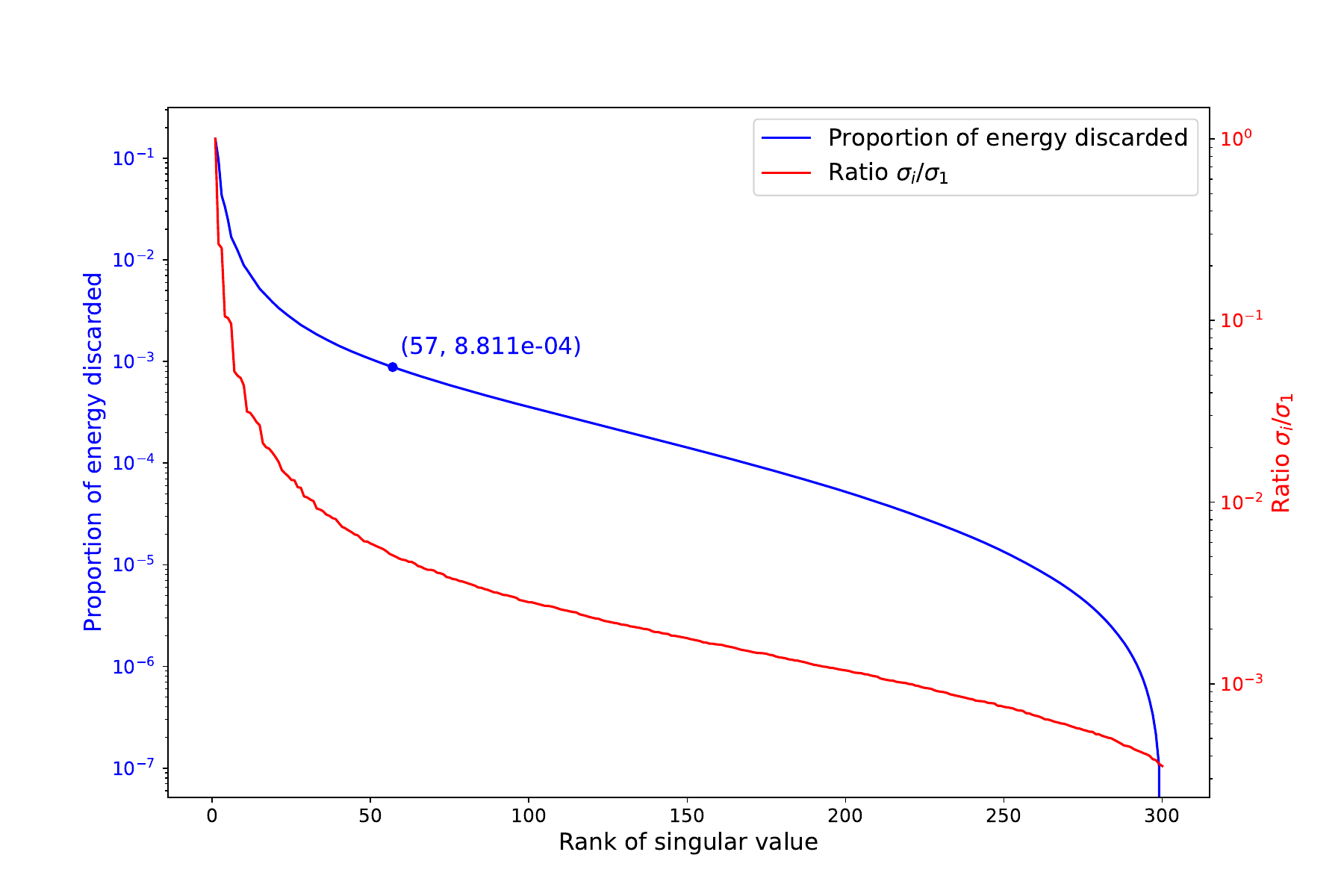}
        \caption{Static Schrödinger equation}
        \label{fig: Schrodinger spectrum}
    \end{subfigure}

    \caption{The ratio of the singular values of the snapshot matrix to the maximum singular value and the corresponding proportion of energy discarded. The dot indicates the reduced dimension and its proportion of energy discarded.}
    \label{fig: spectrum}
\end{figure}

Here, we compare our POD-DNN model with LS-RRBFM in \cite{Chen2016Reduced} and DNN without POD. LS-RRBFM also adopts RBF-FD and an offline-online phase strategy. It constructs the reduced basis by a greedy algorithm. We set the same dimension of reduced basis $d_{\mathrm{Greedy}} = d_{\mathrm{POD}} = d$ for LS-RRBFM and POD. The use of POD significantly reduces the output dimension of DNN from $N = 1036$ to $d = 17$ in this numerical example. We set the same network structure for POD-DNN and DNN, differing only in the dimension of the output layer, where the DNN's output dimension remains the full order $N$ without reduction. 

\autoref{table: Helmholtz error and inference time} presents the average relative $\ell_2$ error and total inference time for different PPDE solvers on a test data set containing $2000$ parameters. The solutions from RBF-FD are taken as the ground truth for error computation. Unlike RBF-FD and LS-RRBFM, POD-DNN and DNN do not require solving any linear system during online inference. Besides, neural networks are very suitable for quick inference given multiple parameters due to parallel computing. They significantly reduce inference time while maintaining accuracy equivalent to traditional methods. POD-DNN gets the lowest error and inference time in this example. We emphasize that all inference computations in this experiment are implemented on the CPU (Intel Core i7-11800H, 2.30 GHz base frequency) for speed comparison, as RBF-FD and LS-RRBFM are unsuitable for parallel computing acceleration on the GPU. The POD-DNN and DNN can further achieve $10$ to $40 \times $ speedup for this test data set when accelerated by our GPU (NVIDIA GeForce RTX 3060 Laptop). \autoref{fig: Helmholtz solution} illustrates the computed solutions and absolute errors for one sample $\mu = (0.927, 1.697)$ in test data set.

\begin{table}[htbp!]
    \centering
    \begin{tabular}{lcc}
        \toprule
        Solvers & Average relative $\ell_2$ error & Inference time(s)\\
        \midrule
        RBF-FD & / & 32.54843\\
        LS-RRBFM & 0.00133 & 0.93357\\
		DNN & 0.00215 & 0.01631\\
        POD-DNN & 0.00068 & 0.00976\\
        \bottomrule
    \end{tabular}
    \caption{The error and inference time on the test set of Helmholtz equation.}
    \label{table: Helmholtz error and inference time}
\end{table}
\begin{figure}[htbp!]
    \centering
    \includegraphics[width=0.7\textwidth]{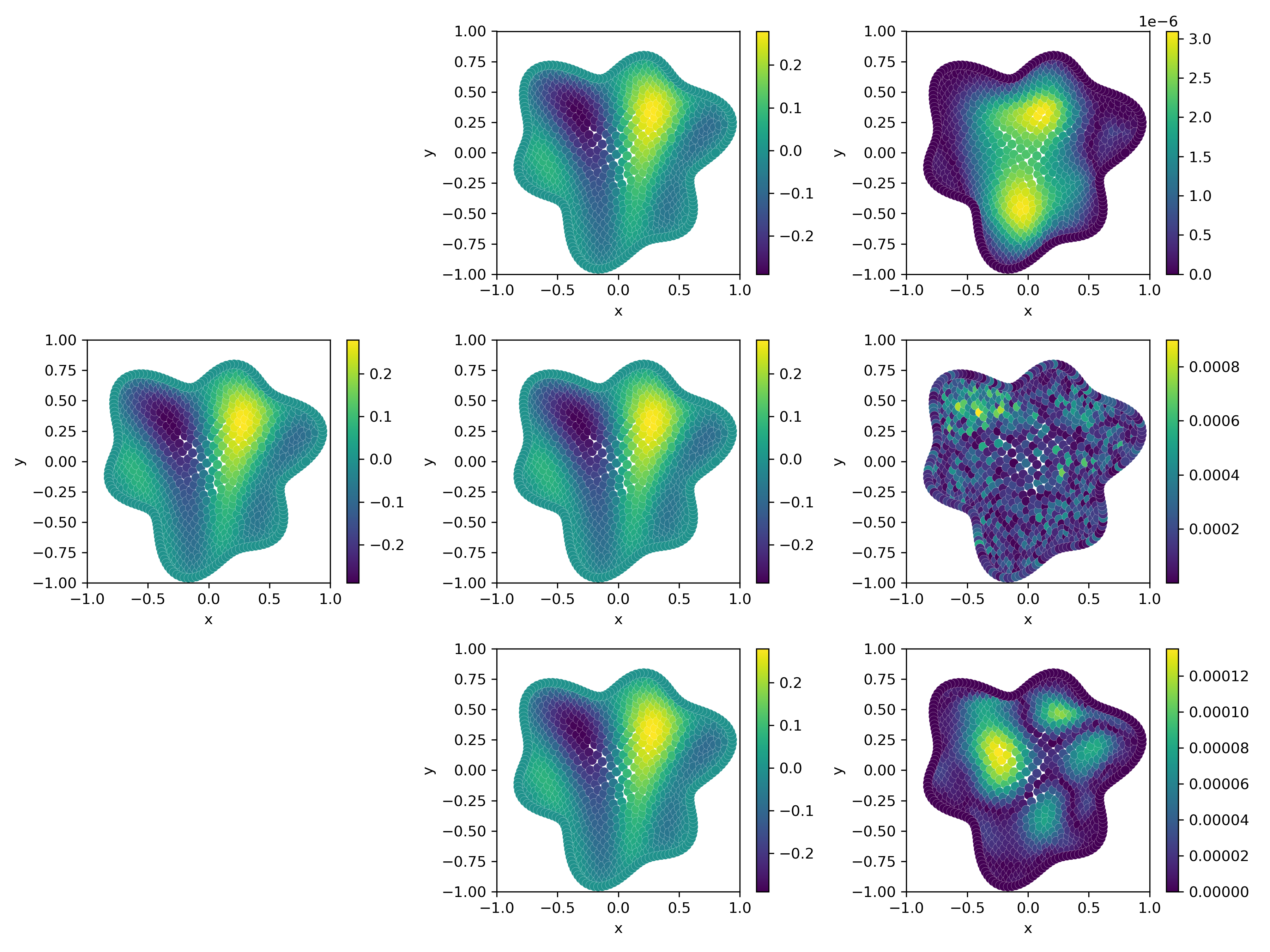}
    \caption{Comparison of solutions of Helmholtz equation. Left: RBF-FD ground truth. Middle: Results from (top to bottom) LS-RRBFM, DNN, and POD-DNN. Right: Their absolute errors.}
    \label{fig: Helmholtz solution}
\end{figure}

\subsection{(2+1)-D sine-Gordon equation}
The discussion in \autoref{section: Preliminaries and Main Results} is based on linear PPDE \eqref{linear PPDE}. Combining iteration methods and time discretization methods, we can also use RBF-FD to solve nonlinear PDEs and time-dependent PDEs. Given these parameters-solutions data generated by RBF-FD, the POD-DNN remains applicable to nonlinear and time-dependent PPDEs. We consider the (2+1)-D sine-Gordon equation in Section 4.2. of \cite{Su2019Numerical}, where they solve this nonlinear and time-dependent PDE with RBF-FD.
\begin{equation}
	\begin{aligned}
		u_{tt} + \beta u_t &= \Delta u + \varrho(x,y) \sin u, && (x,y) \in \Omega = (-7, 7) \times (-7, 7), t \in [0, 10] \\
		u(x,y,0) &= f(x,y), && (x,y) \in \Omega, \\
		u_t(x,y,0) &= 0, && (x,y) \in \Omega, \\
		\frac{\partial u}{\partial n}(x,y,t) &= 0, && (x,y) \in \partial \Omega, t \in [0, 10],
	\end{aligned}
\end{equation}
where 
\[
	f(x,y) = 4 \arctan\left(\exp\left(x+1-2\mathrm{sech}(y+7)-2\mathrm{sech}(y-7)\right)\right).
\]
In \cite{Su2019Numerical}, they fix $\beta = 0.05$ and $\varrho(x, y) \equiv -1$. Here, we treat them as variable parameters in our PPDE formulation:
\[
	\mu = (\beta, \varrho) \in [0.01, 0.1] \times [-2, -0.2].
\]
We follow \cite{Su2019Numerical} to generate the data set.

It should be noted that the snapshot matrix in this problem is actually a 4-order tensor with shape $(N_x, N_y, N_t, n_s)$. The reduced basis may be extracted through general tensor decomposition methods, including higher-order singular value decomposition (HOSVD), tensor train (TT) decomposition, etc. Tensor decomposition methods have proven to be promising ways to tackle high-dimensional PDEs \cite{Bachmayr2016Tensor, Dahmen2016Tensor-Sparsity} and are closely related to reduced-order modeling. The reduced basis studied in this paper stems from the low-dimensional parameters rather than the spatiotemporal structure. For the time-dependent PDE we consider here, the snapshot matrix does not always exhibit reducibility along spatiotemporal dimensions (see a 1-D linear transport equation example presented in Section 6.7 of \cite{Quarteroni2016Reduced}). In summary, we focus exclusively on the reducibility with respect to varying parameters. Therefore, we flatten the spatiotemporal dimensions and perform POD on the resulting $(N_x  N_y  N_t \times n_s)$ matrix. We check the spectral decay of the snapshot matrix in \autoref{fig: sine-Gordon spectrum}. We specify $\varepsilon_{\mathrm{POD}} = 0.01$, leading to a significantly low reduced dimension $d=27$ compared with $N = N_x  N_y  N_t = 328149$.

The LS-RRBFM can not handle time-dependent or nonlinear equations. Therefore, we only compare the POD-DNN and DNN. We find that the coefficient vectors, which are also the output vectors of POD-DNN, exhibit significant scale magnitude variations ranging from $O(1)$ to $O(10^4)$ across different dimensions. This causes some optimization challenges for the neural network, and hence we apply the z-score normalization to the output vectors in POD-DNN. \autoref{table: sine-Gordon error and training time} demonstrates that POD-DNN has an $8\times$ training speedup without compromising generalization performance. \autoref{fig: sine-Gordon solution} shows the computed solutions and absolute errors of POD-DNN at $t = 2.5, 5.0, 7.5, 10.0$ when $\mu = (0.097, -0.432)$. 

\begin{table}[htbp!]
    \centering
    \begin{tabular}{lcc}
        \toprule
        Solvers & Average relative $\ell_2$ error & Training time\\
        \midrule
		DNN & 0.03590 & 124 min \\
        POD-DNN & 0.03718 & 15 min \\
        \bottomrule
    \end{tabular}
    \caption{The error and training time of the sine-Gordon equation.}
    \label{table: sine-Gordon error and training time}
\end{table}
\begin{figure}[htbp!]
    \centering
    \includegraphics[width=\textwidth]{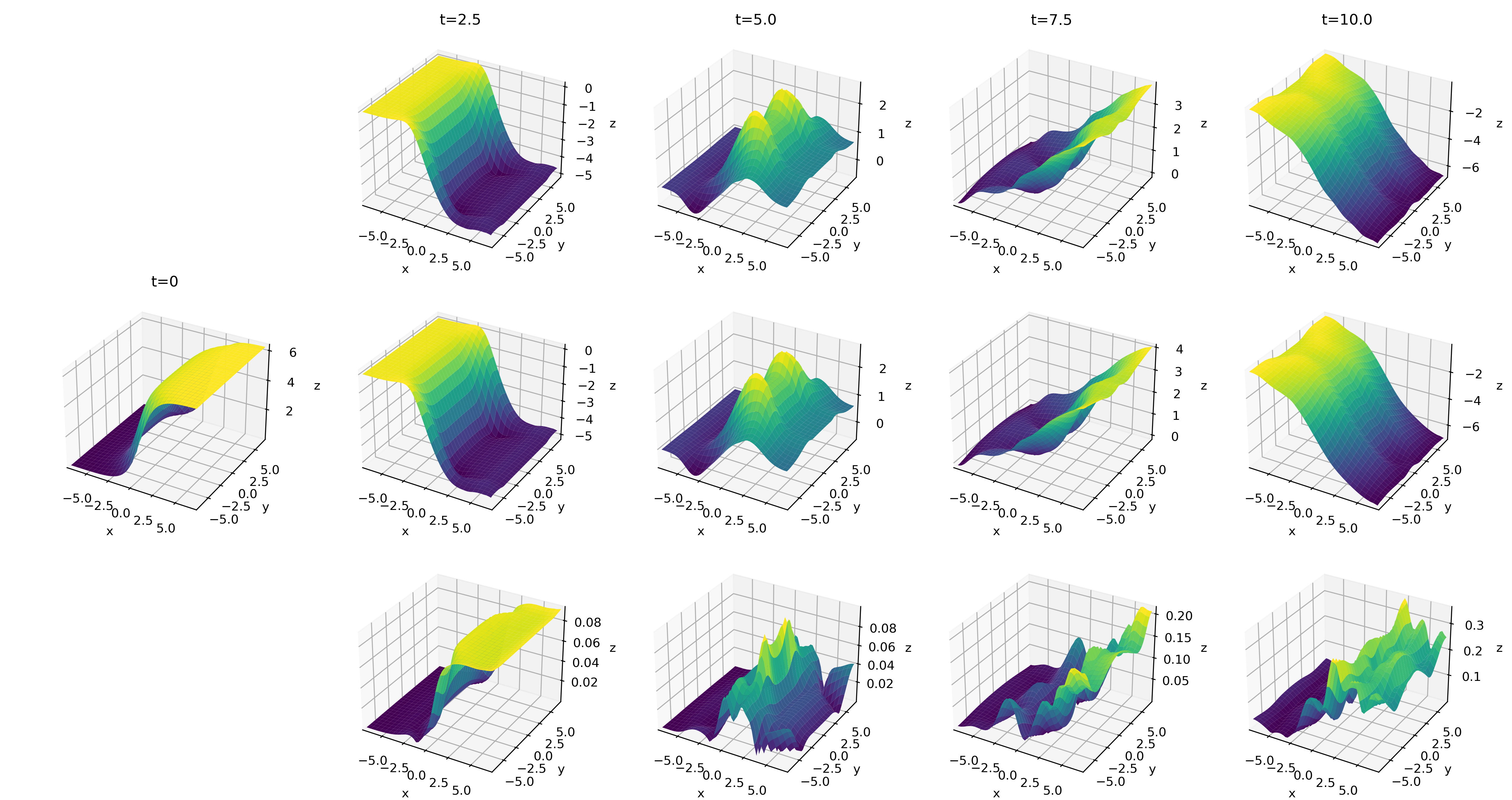}
    \caption{Comparison of solutions of the sine-Gordon equation. Top: POD-DNN. Middle: Initial value $f$ and RBF-FD. Bottom: Absolute errors.}
    \label{fig: sine-Gordon solution}
\end{figure}

\subsection{3-D static Schrödinger equation}
In this example, we solve PPDEs with a high-dimensional parameter space. We consider static Schrödinger equations in a unit 3-D ball with varying boundary conditions.
\begin{equation}
	\begin{aligned}
		-\Delta u + u &= 0, && \mathbf{x} \in \Omega, \\
		u &= f, && \mathbf{x} \in \partial\Omega.
	\end{aligned}
\end{equation}
Here, $\Omega$ is the unit open ball in $\mathbb{R}^3$ and $f$ is a functional parameter on the sphere. We learn the map $f \mapsto u$, which is between infinite-dimensional spaces of functions. Thus, this problem falls within the framework of operator learning and can be addressed by neural operator methods. Neural operators typically employ discretization methods to transform the infinite-dimensional function learning problem to a classical finite-dimensional tensor learning setting. For example, DeepONet \cite{Lu2021Learning} uses the function values at finite many sensors to represent the functions, while FNO \cite{Li2021Fourier} uses a collection of truncated Fourier modes. 

In this experiment, we adopt the same discretization scheme as DeepONet. First, we sample $N_I$ interior nodes and $N_B$ boundary nodes using a robust node generation algorithm proposed in \cite{Shankar2018Robust}. This algorithm is applicable to irregular 2D and 3D domains and is suitable for RBF-FD calculations. Then, we reformulate this operator learning problem as a high-dimensional PPDE with the associated discretized parameter $\mu = (f(\mathbf{x}_1), \cdots, f(\mathbf{x}_{N_B}))$. Finally, we sample functional parameters $f$ in a Gaussian random field $\mathcal{N}(0, (-\Delta_0 + I)^{-2})$, where $\Delta_0$ is the Laplace-Beltrami operator on sphere, and evaluate their values on boundary nodes, which can be done using Karhunen-Loeve expansion and sphere harmonics. \autoref{fig: Schrodinger discretization} illustrates the discretization and 3 distinct cases of $f$. 

\begin{figure}[htbp!]
    \centering
    \includegraphics[width=0.6\textwidth]{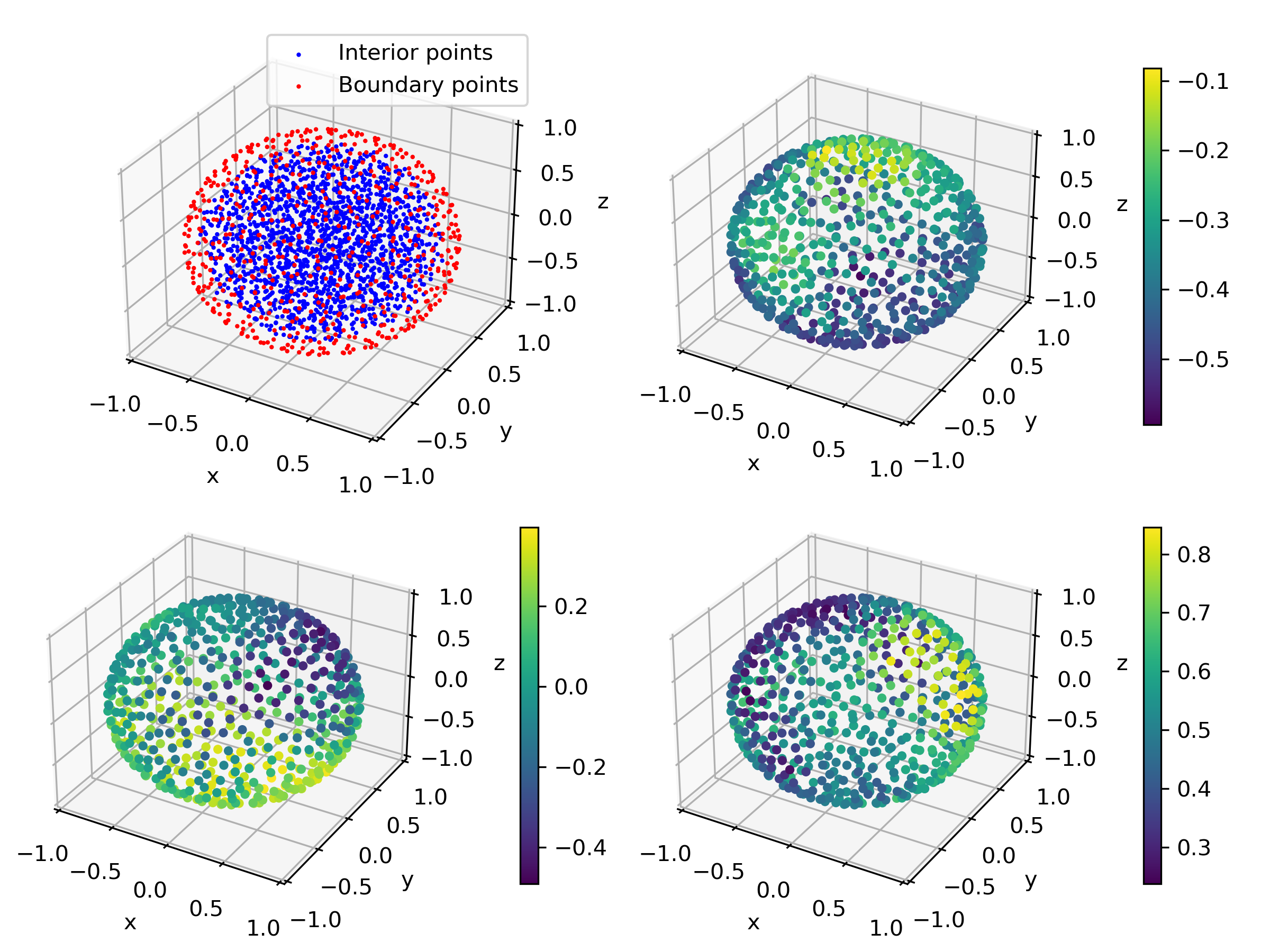}
    \caption{The discretization and 3 distinct cases of $f$ of the static Schrödinger equation.}
    \label{fig: Schrodinger discretization}
\end{figure}

Again, we verify the spectral decay in \autoref{fig: Schrodinger spectrum}. We find that this example exhibits the slowest spectral decay speed, possibly because its parameter space has the highest dimensions among the three PPDEs considered in this paper. Choosing $\varepsilon_{\mathrm{POD}} = 0.03$ results in $d = 57$. We compare numerical results of POD-DNN, DNN, and DeepONet in \autoref{table: Schrodinger error and inference time} and conclude that our algorithm can also solve high-dimensional parameter PDEs and operator learning problems with accuracy comparable to neural operator methods. \autoref{fig: Schrodinger solution} illustrates the computed solution and absolute error of POD-DNN for an $f$ in the test data set. 

\begin{table}[htbp!]
    \centering
    \begin{tabular}{lcc}
        \toprule
        Solvers & Average relative $\ell_2$ error & Inference time(s)\\
        \midrule
		DNN & 0.09291 & 0.00052 \\
        POD-DNN & 0.09908 & 0.00028 \\
		DeepONet & 0.10700 & 0.00076 \\
        \bottomrule
    \end{tabular}
    \caption{The error and inference time of the static Schrödinger equation.}
    \label{table: Schrodinger error and inference time}
\end{table}
\begin{figure}[htbp!]
    \centering
    \includegraphics[width=\textwidth]{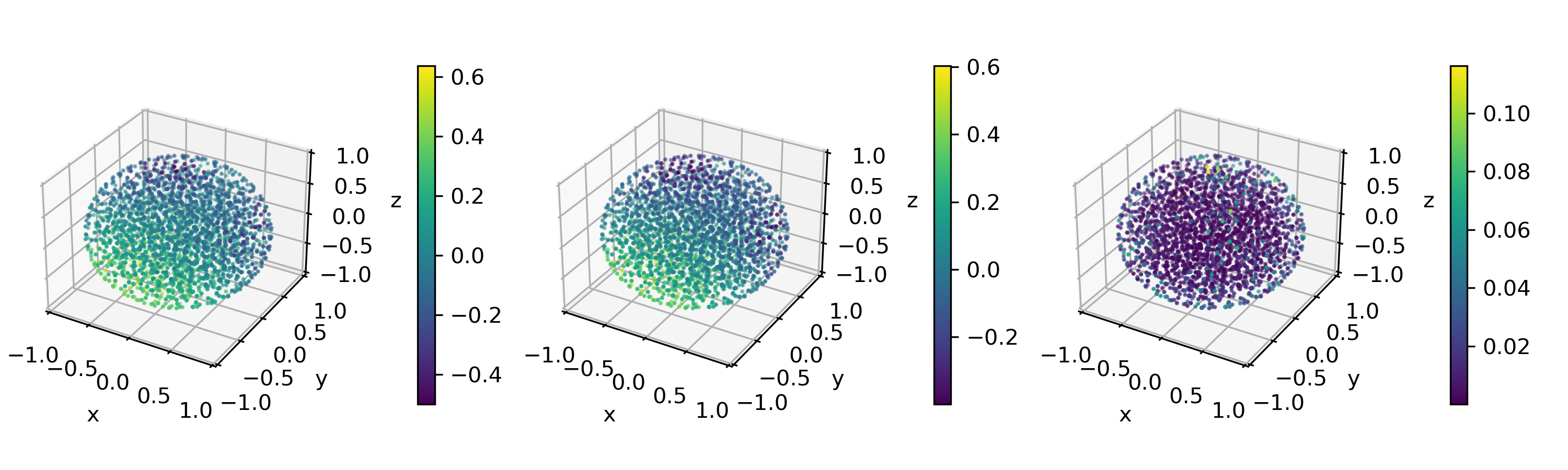}
    \caption{Comparison of solutions of the static Schrödinger equation. Left: RBF-FD. Middle: POD-DNN. Right: Absolute errors.}
    \label{fig: Schrodinger solution}
\end{figure}

Before concluding \autoref{section: Numerical Experiments}, we briefly examine whether these 3 classes of PPDEs satisfy the neural network approximation hypotheses \autoref{assumption: mu to B} and \autoref{assumption: mu to fg} in \autoref{section: Preliminaries and Main Results}. Helmholtz equation and the static Schrödinger equation are linear PDEs and also PPDEs with linear dependence on the parameters. Hence, as we have discussed in \autoref{remark: assumption}, they clearly satisfy \autoref{assumption: mu to B} and \autoref{assumption: mu to fg} with zero approximation error $\varepsilon$ and simple shallow neural networks. For the sine-Gordon equation, although it is a nonlinear equation, the RBF-FD solver uses time discretization and solves a system of linear equations on each time step. Each coefficient matrix depends linearly on the parameters, but the right-hand side vector has an implicit nonlinear composite dependency on the parameters. We need to leverage non-trivial neural network approximation tools to provide a constructive verification for \autoref{assumption: mu to B} and \autoref{assumption: mu to fg}, which may be tedious and is beyond the scope of this paper.

\section{Conclusions}
In this paper, we approximate the parametric maps of PPDEs with neural networks under an RCM framework. Given an approximation accuracy, we establish the upper bounds for the depth and the number of non-zero parameters of the ReLU DNNs. The main result demonstrates that these bounds only depend on the given accuracy and dimension of the reduced basis, which provides a theoretical foundation for the effectiveness of combining RCM and DNN. In numerical experiments, we propose the POD-DNN algorithm for solving PPDEs. Our experiments cover PPDEs with diverse scenarios including irregular geometries, time-dependent PDEs, nonlinear PDEs, and high-dimensional parameters. POD-DNN achieves accuracy comparable to other algorithms while accelerating both training and inference processes. For future research, both in computational and theoretical aspects, different structures of neural networks, like convolutional neural networks, could be designed to solve the parametric system. Meanwhile, neural network approaches for solving PPDEs with nonlinear parameter dependency require further exploration.

\section*{Appendix}
\appendix
\section{RBF-FD Method for PDE}\label{appendix section: RBF-FD Method for PDE}
In this section, we review the RBF-FD method for approximating linear differential operators in \eqref{linear PPDE}. Readers may refer to \cite{Wendland2004Scattered} for definitions of RBF and conditionally positive definite function. 

Let $\ell \in \mathbb{N}$ and 
\[
    \Psi_{\ell-1}:=\mathrm{span}\left\{p: \mathbb{R}^n \to \mathbb{R} \mid p(x)=\prod_{i=1}^n x_i^{\alpha_i}, \alpha_i \in \mathbb{N}\cup \{0\}, \sum_{i=1}^n \alpha_i \leq \ell-1\right\}
\]
denotes the space of $n$-variate polynomials of degree at most $\ell-1$. We consider a conditionally positive definite functions $\varphi: [0, \infty) \to \mathbb{R}$ of order $\ell$ and a $\Psi_{\ell-1}$-unisolvent set of nodes $X = \left\{x_i\right\}_{i=1}^N \subset \Omega$, which contains interior and boundary nodes, i.e., $X= X_{\Omega}\cup X_{\partial \Omega}$ where $X_{\Omega}=\left\{x_1, \ldots, x_{N_I}\right\} \subset \Omega$ and $X_{\partial \Omega}=\left\{x_{N_I+1}, \ldots,x_N\right\} \subset \partial \Omega$. Let $X_i=\left\{x_{i_j}: j=1, \ldots, n_{\mathrm{loc}}^i\right\} \subset X$ be the local set of neighboring points of $x_i$ with $x_{i_1}=x_i$. $X_i$ serves as a stencil for $x_i$. The point $x_i$ is called the center node of the set $X_i$, and $n_{\mathrm{loc}}^i\in \mathbb{N}$ is the stencil size. 

We then consider a linear partial differential operator $\mathcal{L}$ and a sufficiently smooth function $u: \Omega \to \mathbb{R}$. The key idea of the RBF-FD method is to use a weighted sum of the function values at the nodes in the stencil $X_i$ to represent the differential operator applied to the function $u$ and evaluated locally at $x_i$, i.e.,
\begin{equation}\label{RBF-FD weighted sum}
    (\mathcal{L} u)\left(x_i\right) \approx \sum_{j=1}^{n_{\mathrm{loc}}^i} w^i_j u\left(x_{i_j}\right).
\end{equation}
Next, we demonstrate how these weights are computed for typical RBF spaces with (or without) polynomial augmentation. Let 
\[
    \left\{\varphi_{\epsilon, x_{i_j}}: \Omega \to \mathbb{R} \mid \varphi_{\epsilon, x_{i_j}}:= \varphi\left(\epsilon\left\|x-x_{i_j}\right\|\right), \epsilon > 0, j=1, \ldots, n_{\mathrm{loc}}^i\right\} 
\] 
denote the $n_{\mathrm{loc}}^i$ RBFs associated with the shape parameter $\epsilon$ and the nodes of the stencil $X_i$. Define $\displaystyle Q:= \mathrm{dim} \Psi_{\ell-1} = \binom{\ell+d-1}{d}$, and we denote a basis of $\Psi_{\ell-1}$ by $\left\{p_1, \ldots, p_Q\right\}$. RBFs associated with the stencil $X_i$ and augmented by polynomials up to degree $\ell-1$ have the following constrained function space:
\[
    \begin{aligned}
        \mathcal{R}_i:=\bigg\{s: \Omega \to \mathbb{R} \mid s(x) & =\sum_{j=1}^{n_{\mathrm{loc}}^i} \lambda^i_j \varphi_{\epsilon, x_{i_j}}(x)+\sum_{j=1}^Q \widetilde{\lambda}^i_j p_j(x), \lambda^i_j, \widetilde{\lambda}^i_j \in \mathbb{R} \\
           & \text { such that } \sum_{j=1}^{n_{\mathrm{loc}}^i}\lambda^i_j p_k\left(x_{i_j}\right) =0 \text { for all } k \in\{1, \ldots, Q\}\bigg\}.
    \end{aligned}
\]
For any function $u$, we can construct the interpolation function $s_u$ from the space $\mathcal{R}_{i}$, by solving the following linear system
\begin{equation}\label{interpolating linear system}
    \tilde{A}^i
    \begin{bmatrix}
        \lambda^i \\
        \tilde{\lambda}^i
    \end{bmatrix}
    :=
    \begin{bmatrix}
        A^i & P^i \\
        {P^i}^T & \mathbf{0}
    \end{bmatrix}
    \begin{bmatrix}
        \lambda^i \\
        \tilde{\lambda}^i
    \end{bmatrix}
    =
    \begin{bmatrix}
        u^i \\
        \mathbf{0}
    \end{bmatrix}
\end{equation}
where 
\[
    \begin{aligned}
        A^i &:= \left[\varphi\left(\epsilon\left\|x_{i_j}-x_{i_k}\right\|\right)\right]_{jk} \in \mathbb{R}^{n_{\mathrm{loc}}^i \times n_{\mathrm{loc}}^i}, &P^i &:= \left[p_k(x_{i_j})\right]_{jk} \in \mathbb{R}^{n_{\mathrm{loc}}^i \times Q},\\
        \lambda^i &:= \left[\lambda^i_1, \cdots \lambda^i_{n_{\mathrm{loc}}^i}\right]^T, &\tilde{\lambda}^i &:= \left[\tilde{\lambda}^i_1, \cdots \tilde{\lambda}^i_Q\right]^T, \\
        u^i &:= \left[u\left(x_{i_1}\right), \cdots, u\left(x_{i_{n_{\mathrm{loc}}^i}}\right)\right]^T, 
    \end{aligned}
\]
and $\mathbf{0}$ is all-zero vector or matrix. The interpolation problem has a unique solution with the assumption that $\varphi$ is conditionally positive definite of order $\ell$, and the node set $X$ is $\Psi_{\ell-1}$-unisolvent. We define cardinal basis functions $s_1, \ldots, s_{n^{i}_{\mathrm{loc}}} \in \mathcal{R}_i$ w.r.t the stencil notes in $X_i$, i.e., $s_{j}\left(x_{i_k}\right)=\delta_{jk}$, where $\delta_{jk}$ is the Kronecker delta. Then, the interpolation function $s_u$ can be expressed as  
\[
    s_u(x)=\sum_{j=1}^{n^{i}_{\mathrm{loc}}} u\left(x_{i_j}\right) s_j(x).
\]
The differential operator $\mathcal{L}$ on $s_u$ at some $x_i \in X$ is given by
\begin{equation}\label{differential operator on interpolant}
    \mathcal{L} s_u\left(x_i\right) = \sum_{j=1}^{n^{i}_{\mathrm{loc}}} u\left(x_{i_j}\right) \mathcal{L} s_j\left(x_i\right).
\end{equation}
Compared \eqref{differential operator on interpolant} with \eqref{RBF-FD weighted sum}, the weights can be computed by $w^i_j:=\mathcal{L} s_j\left(x_i\right).$ By the linear system \eqref{interpolating linear system} and the definition of $s_j$, we have
\[
    s_j(x)=\left[\varphi_{\epsilon, x_{i_1}}(x), \cdots , \varphi_{\epsilon, x_{i_{{n^{i}_{\mathrm{loc}}}}}}(x), p_1\left(x\right) ,\cdots, p_Q(x)\right] (\tilde{A}^{i})^{-1}
    \begin{bmatrix}
        e_j \\
        \mathbf{0}
    \end{bmatrix},
\]
where $e_j \in \mathbb{R}^n$ denotes the $j$-th unit vector. Hence, define 
\[
    \mathcal{L} \Phi^i = 
    \begin{bmatrix}
        \mathcal{L} \varphi_{\epsilon, x_{i_1}}(x_i), \cdots , \mathcal{L}\varphi_{\epsilon, x_{i_{{n^{i}_{\mathrm{loc}}}}}}(x_i)
    \end{bmatrix}, \quad 
    \mathcal{L} \mathbf{p}^i = 
    \begin{bmatrix}
        \mathcal{L} p_1\left(x_i\right) ,\cdots, \mathcal{L} p_Q(x_i)
    \end{bmatrix},
\]
then we obtain
\[
    \begin{bmatrix}
        w^i_1 , \ldots , w^i_{{n^{i}_{\mathrm{loc}}}}
    \end{bmatrix}
    =
    \begin{bmatrix}
        \mathcal{L} \Phi^i , \mathcal{L} \mathbf{p}^i
    \end{bmatrix} 
    (\tilde{A}^{i})^{-1}
    \begin{bmatrix}
        I_{n^{i}_{\mathrm{loc}}} \\
        \mathbf{0}
    \end{bmatrix},
\]
where $I_{n^{i}_{\mathrm{loc}}} \in \mathbb{R}^{n^{i}_{\mathrm{loc}} \times n^{i}_{\mathrm{loc}}}$ is the identity matrix.

Augmentation of the matrix on the right to an $(n^{i}_{\mathrm{loc}}+Q) \times(n^{i}_{\mathrm{loc}}+Q)$ identity matrix and introduction of dummy variables $\tilde{w}^i_j$ (which may later be discarded since they originated from the arbitrary augmentation to the identity matrix) leads to the linear system of equations
\begin{equation}\label{weights linear system}
    \tilde{A}^{i}
    \begin{bmatrix}
        w^i \\
        \tilde{w}^i
    \end{bmatrix}
    =
    \begin{bmatrix}
        \mathcal{L} \Phi^i \\
        \mathcal{L} \mathbf{p}^i
    \end{bmatrix}
\end{equation}
where 
\[
    w^i:= 
    \begin{bmatrix}
        w^i_1 , \ldots , w^i_{{n^{i}_{\mathrm{loc}}}}
    \end{bmatrix}^T, \quad
    \tilde{w}^i:=
    \begin{bmatrix}
        \tilde{w}^i_1 , \ldots , \tilde{w}^i_Q
    \end{bmatrix}^T
\]
for the desired stencil weights $w^i_j$ in the approximation \eqref{RBF-FD weighted sum}. Without polynomial augmentation, \eqref{weights linear system} simplifies to $A^i w^i=\mathcal{L} \Phi^i$.

Now we apply the RBF-FD method to the discretization of \eqref{linear PPDE} with the Dirichlet boundary condition
\begin{equation}\label{Dirichlet boundary PDE}
    \left\{ 
        \begin{aligned}
            &\mathcal{L} u(x)  =  f(x), \quad  && x \in \Omega,   \\
            &u(x)  =  g(x),  \quad &&x \in \partial \Omega.
        \end{aligned}
    \right.
\end{equation}
For each interior node $x_i \in X_{\Omega}$, we compute a stencil $X_i \subset X$ of size $n^{i}_{\mathrm{loc}}$ and stencil weights $w^i$ by solving \eqref{weights linear system}. The stencil weights are used to compute $u_i$ to approximate  $u\left(x_i\right)$ where $u$ is the solution of \eqref{Dirichlet boundary PDE}, i.e., $u_i \approx u\left(x_i\right)$ at all interior nodes $x_i \in X_{\Omega}$, by enforcing equality in \eqref{RBF-FD weighted sum},
\begin{equation}\label{eq:pde-dis}
    \sum_{j=1}^{n^i_{\mathrm{loc}}} w^i_j u(x_{i_j}) = f\left(x_i\right), \quad i \in\left\{1, \ldots, N_I\right\}.
\end{equation}
From \eqref{Dirichlet boundary PDE}, we have $u_{i_j} = g\left(x_{i_j}\right)$ for all $x_{i_j} \in \partial \Omega$. By substituting the Dirichlet boundary data from \eqref{Dirichlet boundary PDE} into \eqref{eq:pde-dis}, we obtain
\[
    \sum_{\substack{j \in\{1, \ldots, n^i_{\mathrm{loc}}\} \\ \text { s.t. } x_{i_j} \in \Omega}} w_{i_j} u_{i_j}=f\left(x_i\right)-\sum_{\substack{j \in\{1, \ldots, n^i_{\mathrm{loc}}\} \\ \text { s.t. } x_{i_j} \in \partial \Omega}} w_{i_j} g\left(x_{i_j}\right) =: \tilde{f}_i,\quad i \in \{1,\cdots,N_I\}.
\]
Hence, we define the right hand side vector $\tilde{\mathbf{f}}:=\left(\tilde{f}_1, \ldots, \tilde{f}_{N_I}\right)^T \in \mathbb{R}^{N_I}$ and the global stiffness matrix $L^{I} \in \mathbb{R}^{N_I \times N_I}$, containing the row-wise weights of the interior nodes of the $i$-th stencil,
\[
    L^{I}_{i, k}:=\left\{\begin{aligned}
        w_{i_j} && &\exists j \in\{1, \ldots, n^i_{\mathrm{loc}}\} \text { such that } x_{k}=x_{i_j} \text { is in the stencil } X_i, \\
        0 && &\text {otherwise}.
    \end{aligned}\right.
\]
Then, the approximation can be computed by the following linear system
\[
    L^{I}\mathbf{u}_{I} = \tilde{\mathbf{f}},
\]
where $\mathbf{u}_{I} = \left(u_1, \ldots, u_{N_I}\right)^T$ is the discretized approximated solution on interior nodes. This linear system can also contain the boundary nodes and be expressed in a more compact form 
\begin{equation}\label{RBF-FD system}
    \left[
        \begin{array}{c}
            L \\
            \hline
            \begin{array}{c|c}
            \mathbf{0} & I_{N-N_{I}}
            \end{array}
        \end{array}
    \right]
    \begin{bmatrix}
        \mathbf{u}_{I} \\
        \mathbf{u}_{B}
    \end{bmatrix} 
    = 
    \begin{bmatrix}
        \mathbf{f} \\
        \mathbf{g}
    \end{bmatrix},
\end{equation}
where $I_{N-N_{I}}$ is the identity matrix, $L$ is defined in the same way as $L^I$, 
$\mathbf{u}_{B} = \left(u_{N_I+1}, \ldots, u_{N}\right)^T$, $\mathbf{f} = \left(f(x_1), \ldots, f(x_{N_I})\right)^T$, and $\mathbf{g} = \left(g(x_{N_I+1}), \ldots, g(x_{N})\right)^T.$ The discretized approximated solution $\mathbf{u} = (\mathbf{u}_{I}^T, \mathbf{u}_{B}^T)^T$ is the so-called high-fidelity solution in the context of RBM. We note that \eqref{RBF-FD system} is exactly of the same form as the discretized linear equation system \eqref{linear equation system} of the PPDE \eqref{linear PPDE}.

It is also possible to consider boundary conditions associated with differential operators, including the Neumann boundary condition and other mixed boundary conditions. Near the boundary, the stencils are one-sided and hence fail to approximate the differential operator due to the Runge phenomenon. To address this issue, one may need to add one or several disk layers of ghost nodes outside the domain. The function values at the ghost node can be computed by the Neumann boundary condition and then used for approximating the differential operator of the governing equation. See \cite{Larsson2003numerical, Fornberg2006Pseudospectral, Flyer2016Enhancing}. 

Although we only consider time-independent linear PDEs above, RBF-FD can also be utilized in nonlinear PDEs and time-dependent PDEs. For nonlinear PDEs, RBF-FD can solve the linearized PDE at each iteration of Newton's method \cite{Bayona2017role}. For time-dependent PDEs, RBF-FD can give the approximate solution at scattered points at each time and then be combined with various time discretization methods to obtain the complete numerical solutions. More applications of RBF-FD to PDEs can be found in the review paper \cite{Fornberg2015Solving}.

\section{Tables of Hyperparameters}\label{appendix section: Tables of Hyperparameters}
\begin{table}[htbp!]
    \centering
    \begin{tabular}{ll}
        \toprule
        Notation & Meaning \\
        \midrule
        \textbf{RBF-FD} & \\
		$N$ & The number of discrete points \\
		$N_{I}$ & The number of discrete points in interior \\
		$N_{B}$ & The number of discrete points on boundary \\
		$\varphi_\epsilon$ & The RBF \\
		$\epsilon$ & The shape parameter of RBF \\
		$n_\mathrm{loc}$ & The stencil size \\
		\midrule 
		\textbf{POD} & \\
		$n_s$ & The number of parameters used to calculate \\
		& \quad \quad the snapshot matrix \\
		$\varepsilon_{\mathrm{POD}}$ & The error tolerance for POD \\
		$d$ & Dimension of the reduced basis \\
		\midrule
		\textbf{Neural network} & \\
		$n_{\mathrm{data}} = n_{\mathrm{train}} + n_{\mathrm{val}} + n_{\mathrm{test}}$ & Sizes of total, training, validation, and test\\
		& \quad \quad data set \\
		$n_{\mathrm{epoch}}$ & The number of training epochs \\
		lr & The learning rate \\
		$n_{\mathrm{bs}}$ & Batch size \\
		$L$ & The number of hidden layers \\
		$n_{\mathrm{neurons}}$ & The width of each hidden layers \\
        \bottomrule
    \end{tabular}
    \caption{Notations of hyperparameters.}
    \label{table: notations of hyperparameters}
\end{table}
\begin{table}[htbp!]
    \centering
    \begin{tabular}{lcccccc}
        \toprule
        PDE & $N$ & $N_{I}$ & $N_{B}$ & $\varphi_\epsilon$ & $\epsilon$ & $n_\mathrm{loc}$ \\
        \midrule
        Helmholtz & 1036 & 859 & 177 & IMQ & 3 & 50 \\ 
		Sine-Gordon & 328149 & 305525 & 22624 & IMQ & 1 & 25\\
		Static Schrödinger & 2276 & 1691 & 585 & IMQ & 1 & 50\\
        \bottomrule
    \end{tabular}
    \caption{Hyperparameters of RBF-FD. Here, IMQ are the inverse multiquadrics $\displaystyle \varphi_{\epsilon}(r) = \frac{1}{\sqrt{1 + (\epsilon r)^2}}$. They are positive definite functions, and hence the RBF-FD is implemented without polynomial augmentation.}
    \label{table: hyperparameters of RBF-FD}
\end{table}
\begin{table}[htbp!]
    \centering
    \begin{tabular}{lccc}
        \toprule
        PDE & $n_s$ & $\varepsilon_{\mathrm{POD}}$ & $d$ \\
        \midrule
        Helmholtz & 100 & $10^{-6}$ & 17 \\
		Sine-Gordon & 100 & $0.01$ & 27 \\
		Static Schrödinger & 300 & $0.03$ & 57\\
        \bottomrule
    \end{tabular}
    \caption{Hyperparameters of POD.}
    \label{table: hyperparameters of POD}
\end{table}
\begin{table}[htbp!]
    \centering
    \begin{tabular}{lccccccc}
        \toprule
        PDE & $n_{\mathrm{data}}$ & $n_{\mathrm{epoch}}$ & lr & $n_{\mathrm{bs}}$ & $L$ & $n_{\mathrm{neurons}}$ \\
        \midrule
        Helmholtz & $10000=6000+2000+2000$ & 2000 & $10^{-4}$ & 100 & 2 & 500 \\
		Sine-Gordon & $1024=768+128+128$ & 2000 & $10^{-4}$ & 16 & 2 & 500 \\
		Static Schrödinger & $500=300+100+100$ & 2000 & $10^{-4}$ & 10 & 2 & 500 \\
        \bottomrule
    \end{tabular}
    \caption{Hyperparameters of neural network. We also note that all DNNs are trained with the Adam \cite{Kingma2015Adam} optimizer.}
    \label{table: hyperparameters of neural network}
\end{table}


\bibliographystyle{plain}
\bibliography{references.bib}
\end{document}